\documentclass[11pt]{article}

\usepackage[colorlinks=true,bookmarks=true,linkcolor=blue,urlcolor=blue,citecolor=blue,breaklinks=true]{hyperref}
\usepackage{breakcites}

\usepackage{doi}

\usepackage{amsmath}
\usepackage{amsfonts}
\usepackage{amsthm}
\usepackage{amssymb}

\usepackage{mathtools}

\usepackage{xcolor}
\usepackage{bm} 
\usepackage{dsfont} 

\usepackage{graphicx}
\usepackage{tikz-cd}
\usepackage{tikz}
\definecolor{mycolor1}{rgb}{0.105882,0.619608,0.466667}
\definecolor{mycolor2}{rgb}{0.85098,0.372549,0.00784314}
\definecolor{mycolor3}{rgb}{0.458824,0.439216,0.701961}
\definecolor{mycolor4}{rgb}{0.905882,0.160784,0.541176}
\definecolor{mycolor5}{rgb}{0.4,0.65098,0.117647}
\definecolor{mycolor6}{rgb}{0.65098,0.462745,0.113725}
\definecolor{mycolor7}{rgb}{0.901961,0.670588,0.00784314}
\definecolor{mycolor8}{rgb}{0.4,0.4,0.4}
\definecolor{mycolor9}{rgb}{0.301961,0,0.294118}
\definecolor{mycolor10}{rgb}{0.0313725,0.25098,0.505882}

\usepackage{psfrag}

\usepackage{algorithm}
\usepackage{algpseudocode}
\algdef{SE}[DOWHILE]{Do}{doWhile}{\algorithmicdo}[1]{\algorithmicwhile\ #1}%

\usepackage{mathrsfs} 

\usepackage{matlab-prettifier} 

\newcommand{\one}{\mathbf{1}}

\newcommand{\transpose}{^\top\! }

\newcommand{\inner}[2]{\left\langle{#1},{#2}\right\rangle}
\newcommand{\innersmall}[2]{\langle{#1},{#2}\rangle}
\newcommand{\innerbig}[2]{\big\langle{#1},{#2}\big\rangle}

\newcommand{\trace}{\mathrm{Tr}}

\newcommand{\spann}{\mathrm{span}}

\newcommand{\Proj}{\mathrm{Proj}}

\newcommand{\expectt}[1]{\mathbb{E}\!\left\{{#1}\right\}}
\newcommand{\expecttsmall}[1]{\mathbb{E}\{{#1}\}}

\newcommand{\T}{\mathrm{T}}

\newcommand{\SOt}{{\mathrm{SO}(3)}}

\newcommand{\Rnn}{{\mathbb{R}^{n\times n}}}

\newcommand{\Rdn}{{\mathbb{R}^{d\times n}}}

\newcommand{\Rd}{{\mathbb{R}^{d}}}

\newcommand{\Rdd}{{\mathbb{R}^{d\times d}}}

\newcommand{\reals}{{\mathbb{R}}}

\newcommand{\Rn}{{\mathbb{R}^n}}

\newcommand{\grad}{\mathrm{grad}}
\newcommand{\Hess}{\mathrm{Hess}}

\newcommand{\diag}{\mathrm{diag}}
\newcommand{\D}{\mathrm{D}}

\newcommand{\calM}{\mathcal{M}}

\newcommand{\rank}{\operatorname{rank}}

\newcommand{\Sd}{\mathbb{S}^{d-1}}

\newcommand{\frobnormbig}[1]{\big\|{#1}\big\|_\mathrm{F}}
\newcommand{\frobnormsmall}[1]{\|{#1}\|_\mathrm{F}}

\newcommand{\sqfrobnormsmall}[1]{\frobnormsmall{#1}^2}
\newcommand{\sqfrobnormbig}[1]{\frobnormbig{#1}^2}

\newcommand{\TODOF}[1]{}

\newcommand{\ddiag}{\mathrm{ddiag}}
\newcommand{\dist}{\mathrm{dist}}

\newtheorem{theorem}{Theorem}
\newtheorem{lemma}[theorem]{Lemma}

\newtheorem{corollary}[theorem]{Corollary}
\newtheorem{remark}[theorem]{Remark}

\usepackage[ansinew]{inputenc}
\usepackage[round,compress]{natbib}

\usepackage{sectsty}
\makeatletter\def\@seccntformat#1{\protect\makebox[0pt][r]{\csname the#1\endcsname\hspace{12pt}}}\makeatother

\usepackage{subcaption}

\usepackage{enumitem} 

\usepackage[normalem]{ulem} 

\usepackage{multicol}

\newcommand{\aref}[1]{\hyperref[#1]{A\ref{#1}}}

\usepackage[letterpaper,margin=1in]{geometry}

\title{Synchronization on circles and spheres with nonlinear interactions}

\author{
Christopher Criscitiello\thanks{The Wharton School, University of Pennsylvania, USA. \texttt{crisciti@wharton.upenn.edu}}
\and Quentin Rebjock\thanks{Institute of Mathematics, EPFL, Lausanne, Switzerland. \texttt{quentin.rebjock@epfl.ch}}
\and Andrew D.~McRae\thanks{CERMICS, ENPC, Institut Polytechnique de Paris, CNRS, France. \texttt{andrew.mcrae@enpc.fr}}
\and Nicolas Boumal\thanks{Institute of Mathematics, EPFL, Lausanne, Switzerland. \texttt{nicolas.boumal@epfl.ch}}
}

\date{Compiled \today}

\begin{document}

\maketitle

\begin{abstract}
We consider the dynamics of $n$ points on a sphere in $\Rd$ ($d \geq 2$) which attract each other according to a function $\varphi$ of their inner products.
When $\varphi$ is linear ($\varphi(t) = t$), the points converge to a common value (i.e., synchronize) in various connectivity scenarios:
this is part of classical work on Kuramoto oscillator networks.
When $\varphi$ is exponential ($\varphi(t) = e^{\beta t}$), these dynamics correspond to a limit of how idealized transformers process data, as described by~\citet{geshkovski2025mathtransformers}.
Accordingly, they ask whether synchronization occurs for exponential $\varphi$.

The answer depends on the dimension $d$.
In the context of consensus for multi-agent control, \citet{markdahl2018nsphere} show that for $d \geq 3$ (spheres), if the interaction graph is connected and $\varphi$ is increasing and convex, then the system synchronizes.
We give a separate proof of this result.

What is the situation on circles ($d=2$)?
First, we show that $\varphi$ being increasing and convex is no longer sufficient (even for complete graphs).
Then we identify a new condition under which we do have synchronization on the circle (namely, if the Taylor coefficients of $\varphi'$ are decreasing).
As a corollary, this provide synchronization for exponential $\varphi$ with $\beta \in (0, 1]$.
The proofs are based on nonconvex landscape analysis.

\vspace{3mm}\noindent
\end{abstract}

\section{Introduction}

Motivated by open questions \citet{geshkovski2025mathtransformers} raised in their paper ``\emph{A mathematical perspective on transformers},'' (first posted to arXiv late 2023) we consider gradient flow to \emph{maximize}\footnote{In contrast, running \emph{negative} gradient flow in an attempt to \emph{minimize} $f$ is also an interesting (and quite different) problem, as it leads to configurations of points that are well spread out on the sphere. This is related to Thomson's problem and Smale's 7th problem with $\varphi(t) = -\log(1-t)$.} the following function defined for $n$ points on the unit sphere in $\Rd$:
\begin{align}
    f(x_1, \ldots, x_n) & = \frac{1}{2} \sum_{i = 1}^n \sum_{j = 1}^n w_{ij} \varphi(x_i\transpose x_j^{}), && \|x_1\| = \cdots = \|x_n\| = 1.
    \tag{P}
    \label{eq:f}
\end{align}
The real numbers $w_{ij} \geq 0$ are the weights of a graph ($w_{ij} = w_{ji}$) and the function $\varphi \colon [-1, 1] \to \reals$ is twice continuously differentiable.
We assume $\varphi$ is strictly increasing, so that the global maxima correspond to synchronized states: $x_1 = \cdots = x_n$.
The question is: under what conditions does gradient flow reliably converge to such a state?

This is well studied in the \emph{linear} case ($\varphi(t) = t$) as it is equivalent to synchronization of Kuramoto networks of phases~\citep{kuramoto1975oscillators} and (by extension) on spheres.
Synchronization questions also go by the name ``consensus'' in the context of multi-agent control~\citep{sarlette2009consensus}.
From that literature (see Section~\ref{sec:related}), we expect markedly different behavior between synchronization on circles ($d = 2$) and spheres ($d \geq 3$).
The graph structure matters greatly on circles
whereas, for higher-dimensional spheres, gradient flow converges to a synchronized state from almost every initial configuration, 
as long as the graph is connected.
The key difference between the two cases is that spheres are simply connected but circles are not~\citep{markdahl2021multiplyconnected}.
The results in this paper also reflect this dichotomy.

In the paper of \citet{geshkovski2025mathtransformers}, problem~\eqref{eq:f} arises for a complete graph with unit weights ($w_{ij} = 1$ for all $i,j$) 
and \emph{nonlinear} $\varphi$ set to be
\begin{align}
    \varphi_\beta(t) & = \frac{1}{\beta} e^{\beta t},
    \label{eq:phibeta}
\end{align}
where $\beta > 0$ is an ``inverse temperature'' in the language of statistical physics.
The limit $\beta \to 0$ corresponds to the linear case $\varphi(t) = t$.

That setting materializes through their study of an idealized model of how input data (tokens) are processed in infinitely deep neural networks with an architecture inspired by transformers.
In that context, the interacting particles on the sphere correspond to tokens, and time for the gradient flow corresponds to depth in the network.\footnote{\citet{geshkovski2025mathtransformers} obtain gradient flow on $f$ with $\varphi = \varphi_\beta$ by modeling deep networks composed of self-attention and layer-normalization layers.  In the self-attention layers, replacing the exponential of the softmax with $\varphi'$ gives, in exactly the same way, the gradient flow on $f$ with \emph{general} $\varphi$ (the object of our study).}
This step of modelling discrete layers as continuous time variables is in line with previous literature on modeling residual neural networks as neural ODEs~\citep{neuralODEsChen,E2017,Haber2018}.
Neural ODEs to study transformers were first proposed in~\citep{LutransformerneuralODE2020,duttatransformerneuralODE2021,sandersinkformers2022}.

For random initial configurations, \citet{geshkovski2025mathtransformers} observed in their setting that the points 
always converge to a global maximum of $f$, namely, $x_1 = \cdots = x_n$.
This is sensible since $\varphi(x_i\transpose x_j^{})$ is maximal when $x_i = x_j$, but it is not a foregone conclusion given that spheres are nonconvex.

Appendix~\ref{sec:matlabcode}
includes Matlab code
for readers who wish to explore the landscape of~\eqref{eq:f} and the associated dynamical system with various choices of parameters.
It requires the Manopt toolbox~\citep{manopt}, which is under GNU GPLv3 license.


\subsection*{Via landscapes}

To analyze the typical 
asymptotic behavior of this dynamical system, it is sufficient to understand the structure of some critical points of $f$, rather than tracking the entire dynamics.
Indeed, if $\varphi$ (hence $f$) is real-analytic, then gradient flow converges to a critical point~\citep{lojasiewicz1965ensembles}.
Assuming a uniformly random initialization,\footnote{Or any other absolutely continuous probability measure.} the Hessian at that critical point is almost surely negative semidefinite owing to the center-stable manifold theorem~\citep[Thm.~III.7, Ex.~III.3]{shub1987book}.
(Though classical, this argument is rarely spelled out: see~\citep[App.~A]{geshkovski2025mathtransformers} for welcome details.)

Thus, to confirm that gradient flow almost surely converges to a synchronized state, it is sufficient to show that points where the gradient of $f$ is zero and the Hessian is negative semidefinite are in fact synchronized (in particular, that they are global maxima). 
The gradient and Hessian are defined with respect to the usual Riemannian metric on the sphere: see Appendix~\ref{app:gradhess} for explicit formulas.

\subsection*{Positive answers, obstructions and remaining open questions}

To set the scene,
we first state the following theorem to handle $d \geq 3$.
In particular, it positively answers the questions of~\citet{geshkovski2025mathtransformers} for $d \geq 3$ (see Corollary~\ref{cor:spheres}).
This follows directly from~\citep[Thm.~13]{markdahl2018nsphere}.
We give a short proof in Appendix~\ref{app:spheres} based on a randomized selection of tangent vectors plugged into the Hessian quadratic form, paralleling the proof in~\citep[\S4]{mcrae2023benignlowdim} for synchronization of rotations.

\begin{theorem}[{\protect{Spheres, \citet[Thm.~13]{markdahl2018nsphere}}}] \label{thm:spheres}
    Fix $n \geq 1$. Assume
    \begin{enumerate}
        \item the ambient dimension $d$ satisfies $d \geq 3$,
        \item the undirected graph defined by the weights $w_{ij} \geq 0$ is connected, and
        \item $\varphi'(t) > 0$ and $\varphi''(t) \geq 0$ for all $t \in [-1, 1]$.
    \end{enumerate}
    Then, critical points of $f$ where the Hessian is negative semidefinite are global maxima of $f$.
    In particular: local maxima are global maxima, they are the synchronized
    states ($x_1 = \cdots = x_n$), and (if $\varphi$ is real-analytic) gradient flow
    converges to a synchronized state from almost every initialization.
\end{theorem}
\begin{corollary}\label{cor:spheres}
    Theorem~\ref{thm:spheres} applies for $\varphi(t) = t$ and for $\varphi(t) = \varphi_\beta(t)$ with all $\beta > 0$.
\end{corollary}

Let us go through the assumptions in Theorem~\ref{thm:spheres}.
If $\varphi'$ is not positive ($\varphi$ not monotonously increasing) there may be global maxima that are not synchronized: see $\varphi(t) = t^2$ in Theorem~\ref{thm:quadraticphi} below.
If $\varphi''$ is not nonnegative ($\varphi$ not convex) there may be non-global local maxima: consider a tetrahedron ($d = 3, n = 4$) with $\varphi(t) = t^3 + \frac{1}{10}t$.
%
%
If the graph is not connected, then the claims of Theorem~\ref{thm:spheres} still apply to each connected component separately.

The assumption $d \geq 3$ is more interesting---and indeed necessary.
As mentioned above, in the linear case ($\varphi(t) = t$) there are counterexamples when $d = 2$.
However, these are for \emph{incomplete} graphs (the simplest example is a cycle graph; see, e.g., \cite{townsend2020densenetworks} for many more).
In contrast, the setting in \citep{geshkovski2025mathtransformers} centers on complete graphs.

Remarkably, with nonlinear $\varphi$ and $d = 2$, even a \emph{complete} graph with \emph{unit} weights can harbor spurious local maxima.
In Section~\ref{sec:circles}, we construct a \emph{single} function $\varphi$ which satisfies the assumptions of Theorem~\ref{thm:spheres} yet for which, with $d = 2$, the conclusion of that theorem fails for \emph{all} $n \geq 5$.
\begin{theorem}[Circles] \label{thm:circles}
%
Let $d=2$, and consider the complete graph with $w_{ij} = 1$ for all $i, j$.
For $\varepsilon > 0$ and $\tau \in (-1, 1)$, define
\begin{align*}
    \varphi_{\varepsilon, \tau}(t) = \varepsilon \log\!\left(1 + e^{(t-\tau) / \varepsilon}\right),
\end{align*}
which is real-analytic and satisfies $\varphi_{\varepsilon,\tau}'(t), \varphi_{\varepsilon,\tau}''(t) > 0$ for all $t$.
There exist $\varepsilon$ and $\tau$ such that, for all $n\geq 5$, gradient flow on~\eqref{eq:f} with $\varphi = \varphi_{\varepsilon, \tau}$
and uniformly random initialization converges to a spurious local maximum (not synchronized) with positive probability.
\end{theorem}

Intuitively, Theorem~\ref{thm:circles} relies on the fact that there are functions $\varphi$ and configurations $x_1, \dots, x_n$ such that \eqref{eq:f} \emph{locally} behaves much like an incomplete graph with linear $\varphi$
(in particular, a circulant graph in a ``twisted state'' configuration, which is a well-known source of spurious local maxima in the linear case; see, e.g., \cite{wiley2006syncbasin,townsend2020densenetworks}).

For that proof too, we reduce the claim to a landscape analysis.
Indeed, with the constructed function $\varphi$, the function $f$ is real-analytic with a spurious maximizer (non-synchronized states).
The maximizers of analytic functions are Lyapunov stable with respect to gradient flow owing to the {\L}ojasiewicz inequality (see for example \citep{absil2006stableequilibrium}).
This ensures that gradient flow converges to a spurious maximizer with positive probability.

Still, Theorem~\ref{thm:circles} does not exclude the possibility that the landscape is benign for $d = 2$ when $\varphi = \varphi_\beta$.\footnote{As an example, for fixed $n$, \citet[\S6]{sarlette2009synchcircle} construct a $\varphi$ such that the $n$ particles synchronize as long as the graph is connected, but it is different from $\varphi_\beta$ and it does not work for all $n$. \citet{tron2012so3consensus} take a similar approach on $\SOt$.}
We show that this is indeed true for $0 < \beta \leq 1$ in Section~\ref{sec:smallbeta}.
This follows as a corollary of our next theorem, which requires the Taylor expansion coefficients of $\varphi'$ to be nonincreasing.
That corollary improves on earlier results by \citet{geshkovski2025mathtransformers} which
required $\beta \lesssim 1/n$. 
After our paper appeared on arXiv, \citet{polyanskiy2025synchronizationcircle} proved synchronization on the circle for all $\beta \geq -0.16$, as a particularization of a more general result that is different from ours.
This indicates that the assumptions in Theorem~\ref{thm:taylorphi} are more stringent than necessary.

\begin{theorem} \label{thm:taylorphi}
    Fix $d \geq 2$.
    Assume $\varphi'(t) > 0$ and $\varphi''(t) \geq 0$ for all $t \in [-1, 1]$, and also that 
    \begin{align*}
        \varphi'(t) = \sum_{\ell = 0}^\infty a_\ell t^\ell && \textrm{ for all } && t \in [-1, 1] && \textrm{ with } &&
        a_0 > a_1 \geq a_2 \geq a_3 \geq \cdots.
    \end{align*}
    Let the weights of the graph satisfy $w_{ij} = q_iq_j$ for some $q \in \Rn$ with positive entries (in particular, the graph is complete).
    Then all the same conclusions as in Theorem~\ref{thm:spheres} apply.
    The assumption $a_0 > a_1$ can be relaxed to $a_0 \geq a_1$ if $\varphi''(t) > 0$ for all $t \in [-1, 1]$.
\end{theorem}

\begin{corollary}[Small $\beta$] \label{cor:smallbeta}
    Theorem~\ref{thm:taylorphi} applies for $\varphi(t) = t$ and for $\varphi(t) = \varphi_\beta(t)$ with $0 < \beta \leq 1$.
    It also applies for $\varphi(t) = -\log(1-t+\epsilon)$ for all $\epsilon > 0$ (Thomson's problem corresponds to $\epsilon = 0$).
\end{corollary}

In place of the condition on the Taylor coefficients of $\varphi'$ in Theorem \ref{thm:taylorphi}, a variation of that theorem also holds under a similar condition on the Fourier coefficients of $\theta \mapsto \varphi'(\cos\theta)$; this condition is also satisfied (in particular) for $\varphi=\varphi_\beta$ with $0<\beta\leq 1$.

\citet{geshkovski2025mathtransformers} also argue that synchronization occurs almost surely for $\beta \gtrsim n^2$.
We repeat their argument in Appendix~\ref{sec:largebeta} with minor changes to handle an arbitrary connected graph $W$ and more general $\varphi$, and to make all quantities explicit.

\begin{theorem} \label{thm:largebeta}
    Fix $n \geq 1$ and $d = 2$.
    Assume $\varphi'(t) > 0$ for all $t \in [-1, 1]$, and also that
    \begin{align}
    	t \varphi'(t) - (1-t^2) \varphi''(t) < 0 && \textrm{ for all } && -1 \leq t < \cos\!\Big(\frac{\pi}{n}\Big).
        \label{conditiononvarphiforlargebeta}
    \end{align}
    Assume the graph given by weights $w_{ij} \geq 0$ is connected.
    Then all the same conclusions as in Theorem~\ref{thm:spheres} apply.
    Condition~\eqref{conditiononvarphiforlargebeta} holds in particular when $\varphi''(t) \geq \frac{n^2}{\pi^2} \varphi'(t)$ for all $t \in [-1,1]$.
\end{theorem}
%
\begin{corollary}[{\protect large $\beta$, \citet{geshkovski2025mathtransformers}}] \label{cor:largebeta}
    Fix $n \geq 1$ and $d \geq 2$.
    Let $\varphi = \varphi_\beta$ with $\beta \geq \frac{n^2}{\pi^2}$, and consider a connected graph with weights $w_{ij} \geq 0$.
    Then all the same conclusions as in Theorem~\ref{thm:spheres} apply.
\end{corollary}

It remains open whether the same holds for values of $\beta$ between 1 and $\frac{n^2}{\pi^2}$ when $d = 2$ and the graph is merely connected (as opposed to the more restrictive assumptions above).

\begin{remark}
To model transformers, \citet{geshkovski2025mathtransformers} consider two dynamical systems: SA (for self-attention) and USA (for unnormalized self-attention).
Problem~\eqref{eq:f} corresponds to USA.
They note that the two models correspond to gradient flow on the same energy (namely, $f$ with $\varphi = \varphi_\beta$) only with two different Riemannian metrics on the product of spheres.
Whether a point is critical or not is independent of the Riemannian metric, and the same goes for the definiteness of the Hessian at critical points.
Thus, landscape analyses apply directly to both models and provide the same conclusions regarding limit points of the underlying dynamical system.
\end{remark}

Finally, we use a different proof in Appendix~\ref{sec:quadraticphi} to show that the landscape is benign in the quadratic case ($\varphi(t) = \frac{1}{2}t^2$, complete graph), although the global maxima are synchronized only up to sign.
Notice that $\varphi'(t) = t$ can be negative on $[-1, 1]$, hence this falls outside the scope of the main theorems above.
Since $\varphi$ has two maxima on the interval $[-1, 1]$, namely, $\pm 1$, the global maxima of $f$ correspond to points $x_1, \ldots, x_n \in \Sd$ all equal \emph{up to sign}.
The theorem below states that this maximization landscape is benign too.\footnote{In this scenario, problem~\eqref{eq:f} amounts to maximizing $f(X) = \frac{1}{4} \sqfrobnormsmall{X\transpose X}$ (Frobenius norm) where the columns of $X$ have unit norm. In contrast, the \emph{minimizers} of this problem are unit-norm tight frames. Remarkably, there are no spurious local minimizers either: see~\citep{benedetto2003finite} and~\citep{mixon2023benedettofickus}.}

\begin{theorem} \label{thm:quadraticphi}
    Fix $d \geq 2$ and $n \geq 1$.
    Let $\varphi(t) = \frac{1}{2} t^2$ and consider a complete graph with unit weights $w_{ij} = 1$.
    Then, critical points of~\eqref{eq:f} where the Hessian is negative semidefinite are global maxima of~\eqref{eq:f}.
    In particular: local maxima are global maxima, they are synchronized \emph{up to sign} ($x_i = \pm x_j$ for all $i,j$), and gradient flow converges to such a state from almost every initialization.
\end{theorem}

\section{Further related work} \label{sec:related}

Problem~\eqref{eq:f} is closely connected to the Kuramoto model for a network of coupled oscillators \citep{kuramoto1975oscillators,acebron2005kuramoto}, which has deep roots in the dynamical systems literature.
The ``homogeneous'' variant considers the following dynamics for $n$ time-varying angles $\theta_1, \dots, \theta_n$:
\begin{equation}
	\label{eq:kuramoto}
	\dot{\theta}_i = - \sum_{j = 1}^n w_{ij} \sin(\theta_i - \theta_j).
\end{equation}
This is precisely the gradient flow of \eqref{eq:f} in the linear case ($\varphi(t) = t$) with $d = 2$ under the change of variable $x_i = (\cos \theta_i, \sin \theta_i)$.
A basic question is: \emph{which graphs} 
have the property that the system converges to the synchronized state $\theta_1 = \cdots = \theta_n$ from almost every initial configuration?
The literature is vast: see the references below and the survey by \citet{doerfler2014synchronization} (particularly \S5). 
In the physics literature, we can also find studies of how \emph{stable} the synchronized state is, e.g., when the individual systems are not identical~\citep{parastesh2025synchrosimplicial}.
There is also work on synchronization of oscillators showing that, with higher-order interactions, networks can synchronize even with negative coupling~\citep{kovalenko2021contrarianssync}. 
For our purposes, the key results are the following.

For complete graphs and, more generally, sufficiently dense or expander-like graphs, the dynamical system \eqref{eq:kuramoto} synchronizes from almost every initialization \citep{sepulchre2007stabilization,taylor2012kuramoto,kassabov2021densekuramoto,abdalla2022expandersync}.
In contrast,
for sparse or structured connected graphs, the dynamics~\eqref{eq:kuramoto} have stable equilibria other than the synchronized state (equivalently, \eqref{eq:f} with $d = 2$ and linear $\varphi$ has spurious local maxima).
A rich source of such spurious configurations consists of ``twisted states'' on a circulant or otherwise ring-like graph \citep{wiley2006syncbasin,canale2015equilibria,townsend2020densenetworks,yoneda2021synclowerbound}.
One can also construct more exotic counterexamples, including graphs with manifolds of stable equilibria of arbitrary dimension \citep{sclosa2023kuramoto}.
%

In higher dimensions ($d \geq 3$),
less attention has been given to this problem.
Relevant works for synchronization on spheres include
\citep{olfatisaber2006sphereswarms,li2014controlspheres,caponigro2015sphereconsensus,lageman2016sphereconsensus},
though none of these guarantee synchronization almost surely. 
The key work for us is by \citet{markdahl2018nsphere},
who show that a broad class of consensus algorithms on the sphere succeed as long as the interaction graph is connected.
Hence, synchronization/consensus on spheres is fundamentally simpler than on circles.
One may also entertain synchronization on more general manifolds~\citep{sarlette2009consensus,markdahl2021multiplyconnected}.
The rotation groups are of particular interest in applications (and they constitute another way to generalize circles). 

If we minimize rather than maximize $f$, we obtain a packing problem.
These have been extensively studied in the literature (e.g., \citet{cohn2007universally}).
Packing on the circle or sphere is closely related to Smale's 7th problem and  Thomson's problem, which ask for the minimal energy configurations of charges constrained to lie on a circle or sphere.
Among this line of work, most relevant to us is the work of \citet{cohn1960chargescircle}, who considers Thomson's problem on the circle.
Using Morse theory, \citet{cohn1960chargescircle} completely characterizes the minima and critical points of $f$, and their signatures, when $d=2$ and $\varphi(t) = -\log(1-t)$ (and for other similar $\varphi$).
It is unclear to us how to apply the techniques of \citet{cohn1960chargescircle} to the present setting: crucially with $\varphi(t) = -\log(1-t)$, the Hessian of $f$ on the circle corresponds to a Laplacian with all \emph{nonpositive weights}.
This is a substantial departure from our setting:
see the open questions in Section~\ref{conclusions}.


Closer to the transformers literature, see also work by \citet{karagodin2024clustering} which aims to go beyond setting the key, query and value matrices to identity (as done in the parts of~\citep{geshkovski2025mathtransformers} that lead to the model under scrutiny here).
See also the work of \citet{rodriguezabella2024asymptoticattention} which is closely related.
And for a more algorithmic take on the view of deep transformers as \emph{discrete-time} dynamical systems evolving points on a sphere, see work on the nGPT architecture~\citep{loshchilov2024nGPT}.

\section{Riemannian geometry tools and optimality conditions} \label{sec:riemann_basics}
Endow $\Rd$ with the inner product $\inner{u}{v} = u\transpose v$.
The unit sphere $\Sd = \{ x \in \Rd : \|x\| = 1 \}$ is a Riemannian submanifold of $\Rd$: the tangent space $\T_x\Sd = \{ \dot x \in \Rd : x\transpose \dot x = 0 \}$ inherits the Euclidean inner product as a subspace of $\Rd$.
Formally, $f$ in~\eqref{eq:f} is defined on the product manifold
\begin{align*}
    \calM = (\Sd)^n = \{ X \in \Rdn : \diag(X\transpose X) = \one \}
\end{align*}
with the product Riemannian structure.
In this matrix notation, the points $x_1, \ldots, x_n$ are arranged as the columns of $X$, $\diag \colon \Rnn \to \Rn$ extracts the diagonal of a matrix, and $\one \in \Rn$ is the all-ones vector (sometimes denoted $\mathbf{1}_n$ if we want to emphasize the dimension).
The cost function is
\begin{align*}
    f(X) = \frac{1}{2} \innersmall{W}{\varphi(X\transpose X)},
\end{align*}
where $\inner{A}{B} = \trace(A\transpose B)$ is the Frobenius inner product, and $\varphi$ applies entrywise ($\varphi(A)_{ij} = \varphi(a_{ij})$).
The symmetric matrix $W \in \Rnn$ holds the graph weights $w_{ij} \geq 0$.

Based on these choices, we can derive expressions for the Riemannian gradient and Hessian of $f \colon \calM \to \reals$, and deduce necessary optimality conditions for~\eqref{eq:f}.
These are standard computations: see Appendix~\ref{app:gradhess}.

Since our results in Theorems~\ref{thm:circles} and \ref{thm:taylorphi} only require proofs for $d = 2$,
we only spell out the conditions for that case here.
This is simpler in part because the tangent space of a circle is one-dimensional, so that the Riemannian Hessian can be expressed as an ordinary $n \times n$ matrix.

The Riemannian Hessian for $d = 2$ exhibits a Laplacian structure, defined as follows.
Given a symmetric matrix $M \in \Rnn$, the Laplacian of the associated graph (where we think of $m_{ij}$ as the weight between nodes $i$ and $j$) is
\begin{align}
	L(M) & = \diag(M\one) - M,
	\label{eq:laplaciandef}
\end{align}
where $\diag \colon \Rn \to \Rnn$ forms a diagonal matrix.
As a quadratic form, it is well known that $\alpha\transpose L(M) \alpha = \frac{1}{2} \sum_{i = 1}^n \sum_{j = 1}^n m_{ij}(\alpha_i - \alpha_j)^2$.
In particular, if the weights $m_{ij}$ are nonnegative, then $L(M) \succeq 0$.
If, furthermore, the graph is connected, then $\ker L(M) = \spann(\one)$.

We now characterize second-order critical points.
In the following, $\ddiag \colon \Rnn \to \Rnn$ sets all off-diagonal entries of a matrix to zero,
    $\odot$ denotes the entrywise (Hadamard) matrix product, and $M \mapsto M^{\odot 2}$ denotes entrywise squaring.
\begin{lemma} \label{lem:socpcircles}
    Let $d = 2$. Assume $\varphi$ is twice continuously differentiable.  
    The eigenvalues of the Hessian of $f$ at $X \in \calM$ are equal to the eigenvalues of $-L(M)$, where $L(M)$ is the Laplacian~\eqref{eq:laplaciandef} for the graph with weights
     \begin{align}
        m_{ij} = w_{ij} h(x_i\transpose x_j^{}) && \textrm{ where } && h(t) = t \varphi'(t) - (1-t^2)\varphi''(t).
        \label{hesseqnintheta}
    \end{align}
    The Riemannian gradient at $X$ is zero (i.e., $X$ is critical) and the Riemannian Hessian at $X$ is negative semidefinite if and only if
    \begin{align}
        X\ddiag(X\transpose X A) = XA \quad\quad &\textrm{and} \quad\quad \ddiag(X\transpose X A) - A \odot X\transpose X \succeq L(K),
        \label{eq:socpcircles_cond} \\
    	\textrm{where} \quad\quad A = W \odot \varphi'(X\transpose X)
    	\quad\quad &\text{and} \quad\quad K = W \odot \varphi''(X\transpose X) \odot \big(\one\one\transpose - (X\transpose X)^{\odot 2}\big).
        \label{eq:A_K_def}
    \end{align}
\end{lemma}

\section{A sufficient condition for synchronization on the circle: Theorem~\ref{thm:taylorphi}} \label{sec:smallbeta}

Given $q \in \Rn$ with positive entries, define the complete graph $W = qq\transpose$.
(The notation $q$ echoes classical work where each $x_i$ is a particle with charge $q_i$ and $f$ is their associated potential~\citep{cohn1960chargescircle}.)
From Theorem~\ref{thm:spheres}, we know that in dimension $d \geq 3$ it is sufficient for $\varphi$ to be a strictly increasing convex function to rule out spurious local maxima for $f$.
From Theorem~\ref{thm:circles}, we also know that this is not sufficient when $d = 2$.
Yet, experimentally, even with $d = 2$ we do not know of spurious maximizers when $\varphi = \varphi_\beta$~\eqref{eq:phibeta}.
Thus, it appears that $\varphi_\beta$ has additional favorable properties.

In this section, we show as much for $\beta$ up to 1.
Specifically, we prove Theorem~\ref{thm:taylorphi}.
Our argument relies crucially on the \emph{Schur product theorem}~\citep[Thm.~5.2.1]{Horn1991}.

To prove Theorem~\ref{thm:spheres} for spheres (see Appendix~\ref{app:spheres}), we chose a \emph{random} direction $\gamma$, and then moved all points $x_i$ in that same direction $\gamma$ using $\dot x_i = (I - x_i x_i\transpose) \gamma$.
This approach works for spheres, but falls short when applied to circles.
It is natural to try to choose the common direction $\gamma$ more purposefully.
Albeit indirectly, the proof below moves all the points in the direction of their \emph{weighted mean} $\gamma = Xq = \sum_{i=1}^n q_i x_i$.  See Remark~\ref{remarkonmovinginmeandirection} at the end of this section.

We start with a lemma showing that we are done if we can show $x_1, \dots, x_n$ lie in a common (closed) hemisphere (we also use this in the proof of Theorem~\ref{thm:largebeta}).
This strengthens existing results for an \emph{open} hemisphere by \citet[Prop.~12]{markdahl2018nsphere}. 
The rank-deficient case (which is treated separately in the proof) resembles familiar results for Burer--Monteiro relaxations (see, for example,~\citet[Thm.~7]{journee2010low}),
but the standard arguments used in that literature do not directly apply because we are \emph{maximizing} a convex function (and not minimizing).

\begin{lemma} \label{lem:hemisphere}
    Assume $\varphi'(t) > 0$ for all $t \in [-1, 1]$
    and that the graph with weights $w_{ij} \geq 0$ is connected.
    Let $X$ be a critical point of $f$ where the Hessian is negative semidefinite.
    If there exists a nonzero $v \in \Rd$ such that $X\transpose v \geq 0$ (entrywise; that is, the points $x_1, \ldots, x_n$ lie in a common closed hemisphere), then $X$ is a global maximum.
    Such a vector $v$ exists if $\rank(X) < d$.
\end{lemma}
See Appendix~\ref{app:hemisphere} for a proof. With this, we can prove the main result of this section.
\begin{proof}[Proof of Theorem~\ref{thm:taylorphi}]
    Fix $d = 2$ (the case $d \geq 3$ holds by Theorem~\ref{thm:spheres}).
    Lemma~\ref{lem:socpcircles} provides
    \begin{align*}
        X\ddiag(X\transpose X A) = XA && \textrm{ and } && \ddiag(X\transpose X A) - A \odot X\transpose X \succeq L(K),
    \end{align*}
    where
    $L(K)$ is the Laplacian~\eqref{eq:laplaciandef}
    of the graph with weights
    $K_{ij} = w_{ij} \varphi''(x_i\transpose x_j^{})(1 - (x_i\transpose x_j^{})^2)$
    and $A = W \odot \varphi'(X\transpose X)$~\eqref{eq:A_K_def}.
    Since $\varphi'' \geq 0$, it follows that $K_{ij} \geq 0$ and
    so $L(K) \succeq 0$.
    
    Multiply the matrix inequality left and right by $X$ and $X\transpose$ to compress it to a $2\times 2$ matrix inequality.
    Crucially, this allows us to use the first-order condition:
    \begin{align}
        XL(K)X\transpose & \preceq X \left( \ddiag(X\transpose X A) - A \odot X\transpose X \right) X\transpose \nonumber\\
        &= X (A - A \odot X \transpose X) X\transpose \nonumber \\
        &= X(A \odot (\one \one\transpose - X\transpose X))X\transpose \nonumber \\
        & = X\left( W \odot \varphi'(X\transpose X) \odot (\one\one\transpose - X\transpose X) \right) X\transpose.
        \label{eq:foosmallbeta}
    \end{align}
    By the assumption on $\varphi'$, it holds that
    \begin{align*}
        \varphi'(t) (1-t) = a_0 - (a_0 - a_1) t - (a_1 - a_2) t^2 - \cdots
    \end{align*}
    for all $t \in [-1, 1]$, with $a_\ell - a_{\ell+1} \geq 0$ for all $\ell \geq 0$.
    In particular,
    \begin{align*}
        \varphi'(X\transpose X) \odot (\one\one\transpose - X\transpose X) & = a_0 \one \one\transpose - (a_0 - a_1) X\transpose X - (a_1 - a_2) (X\transpose X)^{\odot 2} - \cdots.
    \end{align*}
    Plugging this back into~\eqref{eq:foosmallbeta} yields
    \begin{align*}
        a_0 X W X\transpose \succeq  XL(K)X\transpose + X\!\left( (a_0 - a_1) W \odot X\transpose X + (a_1 - a_2) W \odot (X\transpose X)^{\odot 2} + \cdots  \right)\!X\transpose.
    \end{align*}
    The Schur product theorem states that $M, N \succeq 0 \implies M \odot N \succeq 0$.
    Using $X\transpose X \succeq 0$ and the assumption that $W = qq\transpose \succeq 0$ plus the assumption that $a_1-a_2\geq 0, a_2-a_3 \geq 0$ etc., 
    we deduce 
    \begin{align*}
        a_0 XWX\transpose \succeq (a_0-a_1) X( W \odot X\transpose X )X\transpose + XL(K)X\transpose.
    \end{align*}
    Since $W = qq\transpose$, we also have $W \odot X\transpose X = \diag(q) X\transpose X \diag(q)$.
    With $Z = X\diag(q)^{1/2}$, the above can be restated as:
    \begin{align}
        a_0 (Xq)(Xq)\transpose \succeq (a_0 - a_1) (ZZ\transpose)^2 + XL(K)X\transpose \succeq 0.
        \label{eq:matrixineqfoo}
    \end{align}
    The matrix on the left-hand side has rank 1.
    The matrix in the middle is positive semidefinite since $a_0 - a_1 \geq 0$ and $L(K) \succeq 0$.
    
    If $a_0 > a_1$, then the matrix in the middle of inequality~\eqref{eq:matrixineqfoo} must have rank equal to $\rank(Z) = \rank(X)$, using $q_i > 0$ for all $i$.
    The result then follows, because inequality~\eqref{eq:matrixineqfoo} therefore implies $\rank(X) < 2$, and Lemma~\ref{lem:hemisphere} handles $\rank(X) < 2$.
    
    For $a_0 = a_1$, we assume $\varphi'' > 0$ and take a closer look at $L(K)$.
    Since $w_{ij} = q_iq_j > 0$ and $\varphi''(x_i\transpose x_j^{}) > 0$, it follows that $K_{ij} = w_{ij} \varphi''(x_i\transpose x_j^{})(1 - (x_i\transpose x_j^{})^2) = 0$ if and only if $x_i\transpose x_j^{} = \pm 1$.
    We consider two cases.

    First, suppose $\dim \ker(L(K)) > 1$.
    This means the graph of $K$ is disconnected.
   	For $i,j$ in different connected components, $K_{ij} = 0$, so $x_i = \pm x_j$.
   	Since $K$ has at least two nonempty connected components,
   	this implies $x_i = \pm x_j$ for \emph{all} $i,j$ (simply observe that each $x_k$ is in a component that is different from that of $x_i$ or $x_j$).
    Thus, $\rank(X) = 1$ and the result again follows from Lemma~\ref{lem:hemisphere}.
   	
   	Second, suppose $\ker(L(K)) = \spann(\one)$.
   	Choose a unit $v \in \reals^2$ orthogonal to $Xq$.
   	Then \eqref{eq:matrixineqfoo} implies
   	\[
   		0 \geq (X\transpose v)\transpose L(K) (X\transpose v) \geq 0.
   	\]
   	Then we must have $X\transpose v \in \spann(\one)$.
   	In particular, negating $v$ if necessary, we have $X\transpose v \geq 0$ entrywise,
   	so the result once again follows from Lemma~\ref{lem:hemisphere}.
\end{proof}

\begin{remark}\label{remarkonmovinginmeandirection}
The proof above 
moves all points in the direction of their weighted mean $X q$.
To see this, note that all cases at the end of the proof are handled (explicitly or implicitly) by hitting~\eqref{eq:matrixineqfoo} with an appropriately chosen vector $v$ proportional to $J X q$,
where $J = \left(\begin{smallmatrix*} 0 & -1 \\ 1 & 0 \end{smallmatrix*}\right)$.
One can verify (see the proof of Lemma~\ref{lem:socpcircles} in Appendix~\ref{app:gradhess})
that this translates to tangent vectors $\dot x_i = (I - x_i^{} x_i\transpose) X q$.
\end{remark}


\section{Obstacles to synchronization on the circle: Theorem~\ref{thm:circles}} \label{sec:circles}

The circle ($d = 2$) stands out among ``spheres'' $\Sd$ in that it is not simply connected: the circle itself (a closed loop) cannot be continuously collapsed into a single point while staying on the circle, whereas, say, the equator can be collapsed to a single point on a sphere.
Accordingly, to construct non-synchronizing scenarios on the circle, it makes sense to entertain configurations of points that ``go around'' the circle.
The simplest such configuration is when all points $x_1, \ldots, x_n$ lie on a regular $n$-gon as follows (see Figure~\ref{fig:figurecircles}): 
\begin{align}
    x_{i} = (\cos\theta_i, \sin\theta_i) && \textrm{ with } && \theta_i = 2\pi\frac{i-1}{n} && \textrm{ for } && i = 1, \ldots, n.
    \label{eq:ngon}
\end{align}

In this section we prove Theorem~\ref{thm:circles}: when $d=2$ there exists a real-analytic $\varphi$ with $\varphi'>0, \varphi''>0$ such that $f(X) = \frac{1}{2} \innersmall{\one_n^{} \one_n\transpose}{\varphi(X\transpose X)}$ has a spurious local maximum for all $n \geq 5$.
Let us sketch the proof.
As a first step, we show that \emph{for each} $n$ there is a $\varphi_n$ such that the regular $n$-gon is a spurious local maximum of $f$ with $\varphi = \varphi_n$ (Lemma~\ref{propphiforeachn} below); this follows from a well-known result for linear synchronization on the circle (Lemma~\ref{propphiforrelu}).
In order to build a single $\varphi$ which works for all $n$ as in Theorem~\ref{thm:circles}, we then fix some integer $m$ (e.g., $m=5$) and distribute $n$ points on the regular $m$-gon with roughly $n/m$ points at each vertex.
This configuration may not be critical for $f$ with $\varphi = \varphi_m$ (e.g., if $n$ is not a multiple of $m$), but we argue that it is close to a spurious critical point (and hence that one exists) provided $n$ is large enough (combine Lemmas~\ref{proprepeatingpoints} and~\ref{lemmaIFT} below).
Finally, we exhibit a $\varphi$ that covers all $n \geq 5$.
Now let us proceed to the full proof.


For the linear synchronization problem on the circle ($\varphi(t) = t, d=2$), it is well known that the regular $n$-gon configuration~\eqref{eq:ngon} (with large enough $n$) is a spurious local maximum when the weight matrix $W$ corresponds to a (circulant) nearest-neighbor graph in which every node is connected to at most $68\%$ of its nearest neighbors.
This observation is due to~\citet{wiley2006syncbasin}.

In our setting, we can see this as taking the complete graph $W = \one\one\transpose$ and the ReLU-type function
\begin{align}
    \varphi(t) = \max\{0, t - \tau\}
    \label{relutypefunction}
\end{align}
with $\tau$ at least $\cos(0.68\pi)$: see Figure~\ref{fig:figurecircles}.
This is valid for large $n$.
To handle all $n \geq 5$, we reduce connectivity and require $\tau \geq \cos(.50\pi) = 0$.
We also require $\tau < \cos(\frac{2\pi}{5})$ to ensure that each point interacts with at least its two nearest neighbors.
The following lemma can be checked with formulas by~\citet[\S5.3]{ling2019synchronization}, who obtain the eigenvalues of the Hessian at the regular $n$-gon in closed form using a Fourier transform.

\begin{figure}
    \centering
    \includegraphics[width=0.4\linewidth]{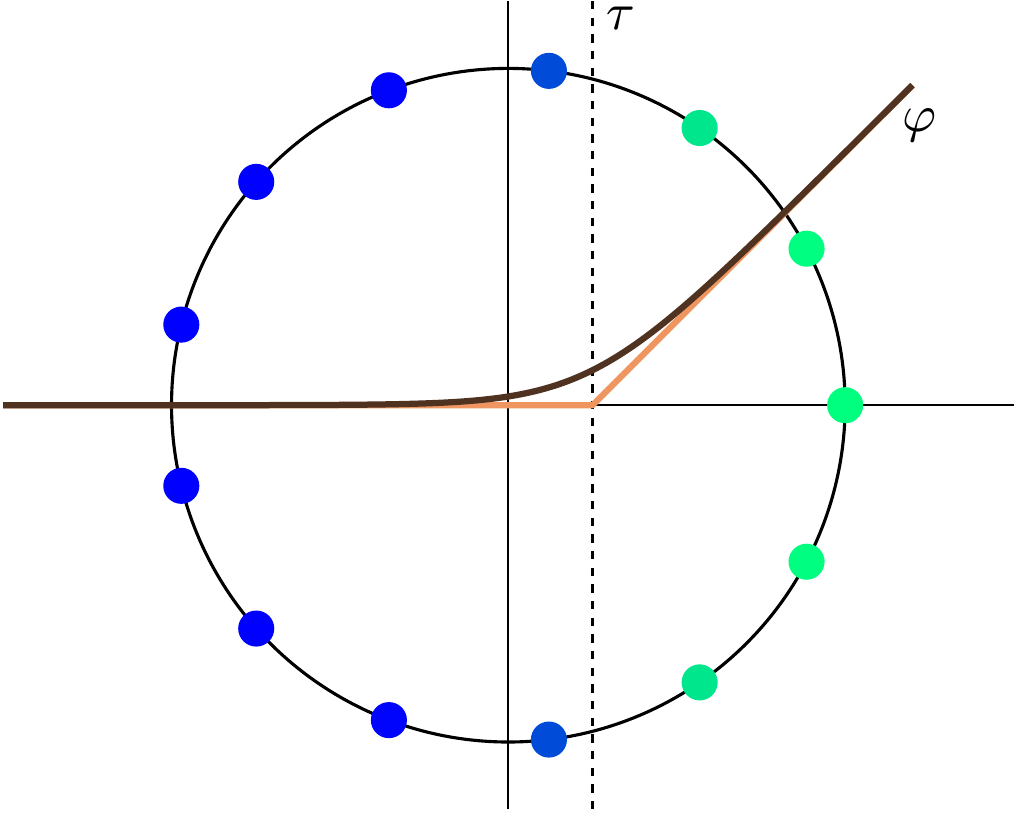}
    \caption{A regular $n$-gon on the circle ($n = 13, d = 2$). In Section~\ref{sec:circles}, we construct the function $\varphi$ by smoothing a ReLU such that $\varphi'(\cos(\theta_i - \theta_j))$ is about 1 when $\cos(\theta_i-\theta_j) > \tau$ and about 0 otherwise. Points are color-coded by $\varphi'(\cos(\theta_i - 0))$.}
    \label{fig:figurecircles}
\end{figure}

\begin{lemma}\label{propphiforrelu}[\citet{wiley2006syncbasin}, \citet[\S5.3]{ling2019synchronization}]
Given $n \geq 5$, let $x_1, \ldots, x_n$ lie on a regular $n$-gon~\eqref{eq:ngon}.
Fix $\tau \in [0, \cos(\frac{2\pi}{5}))$ such that $\tau \not \in \{\cos(2\pi \frac{i-1}{n}) : i = 1, \ldots, n\}$.\footnote{This condition ensures that none of the inner products $x_i \transpose x_j^{}$ for the $n$-gon lie on the kink of the ReLU.}
With
$\varphi$ as in~\eqref{relutypefunction}
and $W = \one_n^{}\one_n\transpose$, the point $X$ is critical for $f$, and $\Hess f(X)$ is negative\footnote{\label{note1}Throughout Section~\ref{sec:circles}, ``negative definite Hessian'' means all eigenvalues of the Hessian are negative except for the single zero eigenvalue which appears due to the global rotation symmetry of the problem.} definite.
\end{lemma}

The next step is simple: we smooth the ReLU~\eqref{relutypefunction} and, by continuity, the regular $n$-gon remains a spurious local maximum.
In order to ensure $\varphi$ is real-analytic, we apply the log-sum-exp smoothing
\begin{align}
  \varphi_{\varepsilon, \tau}(t) = \varepsilon \log(1 + e^{(t-\tau)/\varepsilon})
  \label{eq:phi-softmax}
\end{align}
with parameter $\varepsilon > 0$.

\begin{lemma}\label{propphiforeachn}
  Let $n \geq 5$ and fix $\tau$ as in Lemma~\ref{propphiforrelu}.
  There exists $\varepsilon_{n, \tau} > 0$ such that, for all $\varepsilon \in (0, \varepsilon_{n, \tau}]$, the regular $n$-gon configuration $X$ is critical for $f$ with $W = \one_n^{}\one_n\transpose$ and $\varphi = \varphi_{\varepsilon, \tau}$~\eqref{eq:phi-softmax}.
  Moreover, $\Hess f(X)$ is negative\textsuperscript{\ref{note1}} definite.
\end{lemma}
\begin{proof}[Proof of Lemma~\ref{propphiforeachn}]
  From Lemma~\ref{lem:socpcircles}, one can easily verify by symmetry that the regular $n$-gon $X$ is critical for $f$ for \emph{any} $\varphi$ such that $f$ is differentiable in a neighborhood of $X$. 
  We now study $\Hess f_{\varepsilon, \tau}(X)$, which is well defined (even at $\varepsilon = 0$) since we assume  $\tau \neq \cos(2 \pi \frac{i-1}{n})$ for $i = 1, \ldots, n$.
  Lemma~\ref{propphiforrelu} gives that $\Hess f_{0, \tau}(X)$ is negative\textsuperscript{\ref{note1}} definite.
  Hence, it is enough to show that the curve $\varepsilon \mapsto \Hess f_{\varepsilon, \tau}(X)$ is continuous from the right at $\varepsilon = 0$.
  We can then conclude using the continuity of eigenvalues.
  (More formally, one should remove the trivial zero eigenvalue of the Hessian by fixing one of the points, or passing to the quotient, and then applying this argument.)
  The function $\varphi$ only appears in $\Hess f(X)$ through $\varphi'$ and $\varphi''$ (see Lemma~\ref{lem:gradhess} in Appendix~\ref{app:gradhess}).
  Therefore, it is enough to show that for any fixed $t \neq \tau$, both
  \begin{align*}
    \varepsilon \mapsto \varphi_{\varepsilon, \tau}'(t) = 1 - \big(1 + e^{(t - \tau)/ \varepsilon}\big)^{-1} && \text{and} && \varepsilon \mapsto \varphi_{\varepsilon, \tau}''(t) = \big(2 \varepsilon + 2 \varepsilon \cosh((t - \tau)/ \varepsilon)\big)^{-1}
  \end{align*}
  are continuous from the right at $\varepsilon = 0$.
  This is readily apparent from their expressions.
\end{proof}
Lemma~\ref{propphiforeachn} does not give a $\varphi$ which works for all $n$ as in Theorem~\ref{thm:circles} because, as $n$ grows, it seems that $\varepsilon_{n, \tau}$ must decrease to zero.
However, we can circumvent this issue by the following observation.
Fix $m \geq 5$ and let $n = N m$ be an integer multiple of $m$.
Consider the configuration of points on the circle where $N = n/m$ points are placed at the vertices of a regular $m$-gon---we call this configuration the ``repeated $m$-gon.''
As the $m$-gon is a spurious local maximum for $f$ on $m$ points with $\varphi = \varphi_{\varepsilon_m, \tau}$, it stands to reason that the repeated $m$-gon should be a spurious local maximum on $n$ points with $\varphi = \varphi_{\varepsilon_m, \tau}$ (when $n$ is a multiple of $m$).
This is indeed true.
We give a proof of the following (more general) statement in Appendix~\ref{apprepeatingpoints}.
(Allowing arbitrary $q$, as opposed to just a multiple of $\one$, shall be useful for handling the cases where $n$ is not a multiple of $m$.)
\begin{lemma}\label{proprepeatingpoints}
Assume $\varphi'(1) > 0$.
Let 
$q \in \reals^m$ have positive integer entries and $n = q_1 + \cdots + q_m$.
Given $X = (x_1, \ldots, x_m) \in \reals^{2 \times m}$, define
\begin{align}\label{defofXtildefromX}
  \tilde X = (\underbrace{x_1, \ldots, x_1}_\text{$q_1$ times}, \ldots,\underbrace{ x_m, \ldots, x_m}_\text{$q_m$ times}) \in \reals^{2 \times n}.
\end{align}
If $X$ is critical with negative\textsuperscript{\ref{note1}} definite Hessian for $f_m(X) = \frac{1}{2} \langle q q\transpose , \varphi(X\transpose X) \rangle$, then $\tilde X$ is critical with negative\textsuperscript{\ref{note1}} definite Hessian for $f_n(\tilde X) = \frac{1}{2} \langle \one_n^{} \one_n\transpose , \varphi({\tilde X}\transpose \tilde X) \rangle$.
\end{lemma}

For fixed $m \geq 5$, this shows that $f$ has a spurious local maximum with $\varphi = \varphi_{\varepsilon_m, \tau}$ for $n$ points, where $n$ is a multiple of $m$ (and $\tau$ is chosen as in Lemma~\ref{propphiforrelu} with $n$ replaced by $m$).
To handle a number of points $n = N m + r$ that is not a multiple of $m$ (i.e., $r \in [1, m-1]$), first place $N$ points on each vertex of the regular $m$-gon, then distribute the extraneous $r$ points arbitrarily on the $m$ vertices.
This configuration is unlikely to be a critical point; however, it is close to one: if the total number of points is sufficiently large, the few extra points should only cause a minor perturbation from what was a (strict) spurious local maximum.

In Lemma~\ref{lemmaIFT}, this intuition is made rigorous via \emph{the implicit function theorem}.
In the language of Lemma~\ref{proprepeatingpoints}, we let $X$ be the regular $m$-gon, and we take each $q_i$ to be an integer close to $N$; consequently $\frac{1}{N} q$ is close to $\one_m$.
We rescale by $\frac{1}{N}$ so that $q$ itself becomes close $\one_m$ (this scaling does not affect the landscape of $f$).
The proof is in Appendix~\ref{apprepeatingpoints}.
\begin{lemma}\label{lemmaIFT}
Fix $m \geq 5$ odd.
There exists $\delta > 0$ such that if $\varepsilon, \tau$ and $q \in \reals^m$ satisfy
\begin{align*}
  \|q - \one_m\| \leq \delta, &&
  0 \leq \varepsilon \leq \delta &&
  \textrm{ and } &&
  -\delta \leq \tau \leq \delta,
\end{align*}
then $f(X) = \frac{1}{2} \langle q q\transpose , \varphi_{\varepsilon, \tau}(X\transpose X) \rangle$ has a spurious critical point with negative\textsuperscript{\ref{note1}} definite Hessian.
\end{lemma}

We can now combine Lemmas~\ref{propphiforeachn},~\ref{proprepeatingpoints} and~\ref{lemmaIFT} to prove Theorem~\ref{thm:circles}.

\begin{proof}[Proof of Theorem~\ref{thm:circles}]
Fix $m = 5$.
Write $n = N m + r$ with remainder $r \in [0, m-1]$.
Let $q = (1+r/N, 1, 1, \ldots, 1)\transpose \in \reals^m$.
Note that if $n$ is sufficiently large then $q$ is arbitrarily close to $\one_m$.
Invoke Lemma~\ref{lemmaIFT} to select $\delta > 0$.
Set $n_0 \geq 5$ such that $\|q - \one_m\| \leq \delta$ for all $n > n_0$.
If
\begin{align}\label{constraintsonparams}
  n > n_0, && 0 \leq \varepsilon \leq \delta, && \textrm{ and } && -\delta \leq \tau \leq \delta,
\end{align}
then $f(X) = \frac{1}{2} \langle q q\transpose , \varphi_{\varepsilon, \tau}(X\transpose X) \rangle$ has a spurious critical point with negative\textsuperscript{\ref{note1}} definite Hessian: call it  $X_{q, \varepsilon, \tau}$.

Let $\tilde X_{q, \varepsilon, \tau}$ be the configuration of $n$ points where $N q_i$ points are placed at point $i$ of $X_{q, \varepsilon, \tau}$, as in equation~\eqref{defofXtildefromX}.
By Lemma~\ref{proprepeatingpoints} and our choice of $q$, this configuration on $n$ points is a spurious local maximum for $f(X) = \frac{1}{2}\innersmall{\one_n^{} \one_n\transpose}{\varphi_{\varepsilon, \tau}(X\transpose X)}$.
In other words, we have shown that for all $n, \varepsilon, \tau$ satisfying~\eqref{constraintsonparams}, $f$ has a spurious local maximum with $\varphi = \varphi_{\varepsilon, \tau}$ and the graph $W = \one_n^{}\one_n\transpose$.

It remains to choose $\varepsilon$ and $\tau$ to accommodate $m \leq n \leq n_0$.
Select any $\tau \in [0, \min\{\delta, \cos(\frac{2\pi}{5})\})$ such that $\tau \neq \cos(2 \pi \frac{i-1}{n})$ for all integers $n \in [m, n_0]$ and $i \in [1, n]$.
Further define
$$\varepsilon = \min\{\delta, \varepsilon_{m, \tau}, \varepsilon_{m+1, \tau}, \ldots, \varepsilon_{n_0, \tau}\}$$
where $\varepsilon_{n, \tau}$ are provided by Lemma~\ref{propphiforeachn}.
Then, by Lemma~\ref{propphiforeachn}, we conclude that $f$ with $\varphi = \varphi_{\varepsilon, \tau}$ has a spurious local maximum for all $n \geq m$.
\end{proof}

\section{Conclusions and perspectives}\label{conclusions}
We conclude with a list of open questions for nonlinear synchronization on a circle (complete graph $W = \one\one\transpose$ and $\varphi = \varphi_\beta$~\eqref{eq:phibeta} unless otherwise stated).

\begin{itemize}
\item (\textbf{Exponential case}) For which $\beta$ and which graphs do we have synchronization on a circle with $\varphi_\beta$? For complete graphs and $\beta \geq -0.16$, see~\citep{polyanskiy2025synchronizationcircle}.
    
\item (\textbf{Minimal configurations}) Is the regular $n$-gon the minimal configuration of $f$?
\citet{cohn1960chargescircle} shows that this is true for $\varphi(t) = -\log(1-t)$.  
On the other hand, Theorem~\ref{thm:circles} shows that there exist nice $\varphi$ for which the $n$-gon is not minimal.

\item (\textbf{Critical configurations}) We say a configuration consists of \emph{clusters} if there exist $i \neq j$ such that $x_i = x_j$, i.e., at least two points overlap.
The number of clusters is the number of distinct values of the points $x_i$.
Do all critical configurations of $f$, apart from the regular $n$-gon, consist of clusters?
Numerically this appears to be the case.
\citet{cohn1960chargescircle} shows that this is true for $\varphi(t) = -\log(1-t)$.
Note that for $d \geq 3$ there do exist fully non-clustered critical configurations which are not minimal (e.g., an $n$-gon on the equator of a sphere).

\item (\textbf{Signatures of critical configurations})
The signature of a critical configuration is the number of positive eigenvalues of the Hessian at that configuration.
Is the signature of every critical configuration at least the number of clusters minus one?
Numerically, this appears to be true.
\citet{cohn1960chargescircle} shows that for $\varphi(t) = -\log(1-t)$ the signature always \emph{equals} the number of clusters minus one---this is not true in general for $\varphi = \varphi_\beta$ (e.g., for $n=9, \beta=1$, and a critical configuration with four clusters of sizes $1, 2, 2, 4$).

By Lemma~\ref{proprepeatingpointsrefinement} (which is used to prove Lemma~\ref{proprepeatingpoints} above), to show that the signature is at least the number of clusters minus one, it is enough to show the following: if $q$ has positive entries and $X$ is a critical configuration of $f(X) = \frac{1}{2} \langle q q\transpose, \varphi_\beta(X\transpose X)\rangle$ with distinct points, then all eigenvalues of $\Hess f(X)$ are positive (except for the single zero eigenvalue).

\item (\textbf{Dynamics of gradient flow})
\citet{geshkovski2025mathtransformers} are not only interested in the asymptotic behavior of gradient flow on $f$ but also its dynamics.
Numerically, gradient flow on $f$ often slows down near saddles consisting of only a few clusters (also called meta-stable states \citep{cohn1960chargescircle,erber1991equilibrium}).
Also, gradient flow seems to jump between such saddles (akin to saddle-to-saddle dynamics, see~\citep{jacot2022saddletosaddle,berthier2023saddletosaddle,pesme2023saddletosaddle} and references therein).
Can one characterize these dynamics?
A first step would be to compute the maximal eigenvalue of the Hessian at clustered critical configurations, as this eigenvalue controls the escape time of gradient flow.
How does that eigenvalue depend on the number of clusters?

Recently, progress has been made toward understanding the dynamics of gradient flow on $f$. 
When the number of particles is large, \citet{brunoagazzi1} provide a mathematical justification for the emergence and persistence of clustered configurations. Similarly, under some assumptions, \citet{geshkovski2024dynamicmetastabilityselfattentionmodel} show that the particles remain confined near clustered states for long time scales.
\end{itemize}

\newcommand{\acktext}{%
We thank Pedro Abdalla, Afonso Bandeira and Cyril Letrouit for insightful discussions.

\vspace{1mm}
\noindent
This work was supported by the Swiss State Secretariat for Education, Research and Innovation (SERI) under contract number MB22.00027.%
}
\section*{Acknowledgments}
\acktext{}

\bibliographystyle{abbrvnat}
\bibliography{references}

\appendix

\section{Riemannian Gradient and Hessian for general $d$} \label{app:gradhess}
In this section, we give additional Riemannian geometry background and derive the gradient and Hessian of problem \eqref{eq:f} for general $d \geq 2$.
From this, we can derive the simplified criticality conditions for $d = 2$ as they appear in Lemma~\ref{lem:socpcircles} in Section~\ref{sec:riemann_basics}.
See, for example, \citep{AMS08,boumal2020intromanifolds} for details of such computations in the context of Riemannian optimization.

For convenience, we repeat some definitions from Section~\ref{sec:riemann_basics}.
We endow $\Rd$ with the inner product $\inner{u}{v} = u\transpose v$.
The unit sphere $\Sd = \{ x \in \Rd : \|x\| = 1 \}$ is a Riemannian submanifold of $\Rd$: the tangent space $\T_x\Sd = \{ \dot x \in \Rd : x\transpose \dot x = 0 \}$ inherits the Euclidean inner product as a subspace of $\Rd$.
Formally, $f$ is defined on the product manifold
\begin{align*}
	\calM = (\Sd)^n = \{ X \in \Rdn : \diag(X\transpose X) = \one \}
\end{align*}
with the product Riemannian structure.
In this matrix notation, the points $x_1, \ldots, x_n$ are arranged as the columns of $X$, $\diag \colon \Rnn \to \Rn$ extracts the diagonal of a matrix, $\one \in \Rn$ is the all-ones vector, and
\begin{align*}
	f(X) = \frac{1}{2} \innersmall{W}{\varphi(X\transpose X)},
\end{align*}
where $\inner{A}{B} = \trace(A\transpose B)$ is the Frobenius inner product and $\varphi$ applies entrywise: $\varphi(A)_{ij} = \varphi(a_{ij})$.
The symmetric matrix $W \in \Rnn$ holds the graph weights $w_{ij} \geq 0$.

An important object for studying $\calM$ is the orthogonal projection from $\Rdn$ to the tangent space
\begin{align}
	\T_X\calM & = \{ \dot X \in \Rdn : \diag(X\transpose \dot X) = 0 \},
	\label{eq:TXM}
\end{align}
which has the following formula:
\begin{align}
	\Proj_X(Z) = Z - X\ddiag(X\transpose Z),
	\label{eq:Proj}
\end{align}
where $\ddiag \colon \Rnn \to \Rnn$ sets all off-diagonal entries of a matrix to zero.
Indeed, the $i$th column of $\Proj_X(Z)$ is $z_i - (x_i\transpose z_i^{})x_i$, which is the projection of $z_i$ to the tangent space $\T_{x_i}\Sd$.

With this, we can give the general formulas for the gradient and Hessian of $f$ on $\calM$:
\begin{lemma} \label{lem:gradhess}
	Assume $\varphi$ is twice continuously differentiable.
	Given $X \in \calM$, let
	\begin{align}
		S & = \ddiag(X\transpose X A) - A && \textrm{ with } && A = W \odot \varphi'(X\transpose X),
		\label{eq:SA}
	\end{align}
	where $\odot$ is the entrywise (Hadamard) product and $\varphi'$ applies entrywise.
	The Riemannian gradient of $f \colon \calM \to \reals$ at $X \in \calM$ is given by
	\begin{align}
		\grad f(X) & = -XS.
		\label{eq:gradf}
	\end{align}
	The Riemannian Hessian of $f$ at $X$ is a self-adjoint linear map defined by the quadratic form
	\begin{align}
		\innersmall{\dot X}{\Hess f(X)[\dot X]} & = - \innersmall{\dot X\transpose \dot X}{S} + \frac{1}{2} \inner{(X\transpose \dot X + \dot X\transpose X)^{\odot 2}}{W \odot \varphi''(X\transpose X)} 
		\label{eq:Hessfquadraticform}
	\end{align}
	for all $\dot X$ in the tangent space $\T_X\calM$~\eqref{eq:TXM}, where $M \mapsto M^{\odot 2}$ denotes entrywise squaring.
\end{lemma}

\begin{proof}
	Given $\dot X \in \T_X \calM$, one calculates that the directional derivative of $f$ at $X$ along $\dot X$ is
	\begin{align*}
		\D f(X)[\dot X] & = \frac{1}{2} \inner{W}{\varphi'(X\transpose X) \odot (\dot X\transpose X + X\transpose \dot X)} \\
		&= \inner{X\left(W \odot \varphi'(X\transpose X)\right)}{\dot X} = \innersmall{XA}{\dot X}.
	\end{align*}
	Thus, the Riemannian gradient is the orthogonal projection~\eqref{eq:Proj} of $XA$ to $\T_X \calM$~\citep[Prop.~3.61]{boumal2020intromanifolds}, that is,
	\begin{align*}
		\grad f(X) & = \Proj_X(XA) = XA - X\ddiag(X\transpose X A) = -XS.
	\end{align*}
	
	The directional derivative of $\grad f(X)$ along $\dot X \in \T_X\calM$ is given by
	\begin{align*}
		\D\grad f(X)[\dot X] & = -\dot X S - X\dot S,
	\end{align*}
	where $\dot S$ is the derivative of $S$~\eqref{eq:SA} at $X$ along $\dot X$.
	More explicitly, $\dot S = \dot D - \dot A$, where $\dot D$ is diagonal and $\dot A$ is the derivative of $A$~\eqref{eq:SA} at $X$ along $\dot X$.
	The Riemannian Hessian at $X$ as a quadratic form is given by~\citep[Cor.~5.16]{boumal2020intromanifolds}
	\begin{align*}
		\innersmall{\dot X}{\Hess f(X)[\dot X]} & = \innersmall{\dot X}{\D\grad f(X)[\dot X]} \nonumber\\ & = -\innersmall{\dot X\transpose \dot X}{S} - \innersmall{X\transpose \dot X}{\dot S} \\
		&= -\innersmall{\dot X\transpose \dot X}{S} + \innersmall{X\transpose \dot X}{\dot A},
	\end{align*}
	where the last equality uses $\innersmall{X\transpose \dot X}{\dot D} = 0$, which follows from the fact that $\diag(X\transpose \dot X) = 0$ for any tangent vector $\dot X$.
	
	Finally, note that $\dot A = W \odot \varphi''(X\transpose X) \odot (X\transpose \dot X + \dot X\transpose X)$.
	This is symmetric, so $\innersmall{X\transpose \dot X}{\dot A} = \frac{1}{2} \innersmall{X\transpose \dot X + \dot X\transpose X}{\dot A}$.
	Thus we obtain \eqref{eq:Hessfquadraticform}.
\end{proof}

In Sections~\ref{sec:smallbeta} and \ref{sec:circles},
we are primarily interested in the case $d = 2$.
It is then convenient to particularize Lemma~\ref{lem:gradhess} to $d= 2$,
where the Hessian has a simpler matrix form.
Using Lemma~\ref{lem:gradhess}, let us prove Lemma~\ref{lem:socpcircles}.

\begin{proof}[Proof of Lemma~\ref{lem:socpcircles}]
    Start from Lemma~\ref{lem:gradhess} and recall the definitions $A = W \odot \varphi'(X\transpose X)$ and $S = \ddiag(X\transpose X A) - A$~\eqref{eq:SA}.
    The first-order condition is $0 = \grad f(X) = -XS = XA - X\ddiag(X\transpose X A)$, as stated in \eqref{eq:socpcircles_cond}.
    
    For the second-order condition, we consider the eigenvalues of the Hessian quadratic form (repeated here from \eqref{eq:Hessfquadraticform}):
    \begin{align}
        \innersmall{\dot X}{\Hess f(X)[\dot X]} & = - \innersmall{\dot X\transpose \dot X}{S} + \frac{1}{2} \inner{(X\transpose \dot X + \dot X\transpose X)^{\odot 2}}{W \odot \varphi''(X\transpose X)}.
        \label{eq:Hessrepeatproof}
    \end{align}
    With $d = 2$, \emph{all} tangent vectors $\dot X \in \T_X\calM$~\eqref{eq:TXM} are of the form
    \begin{align*}
        \dot X = J X \diag(\alpha) && \textrm{ with } && J = \begin{pmatrix*} 0 & -1 \\ 1 & 0 \end{pmatrix*} && \textrm{ and } && \alpha \in \Rn.
    \end{align*}
    This is because $Jx_i$ rotates $x_i$ by $\pi/2$, making it a basis for the tangent space to the circle at $x_i$.
    This corresponds to expanding tangent vectors in an orthonormal basis with coordinates $\alpha$, which has no effect on eigenvalues of quadratic forms, so we now express~\eqref{eq:Hessrepeatproof} in terms of $\alpha$.
    
    Using $\dot X\transpose \dot X = \diag(\alpha) X\transpose X \diag(\alpha) = (X\transpose X) \odot \alpha\alpha\transpose$ and also $\diag(X\transpose X) = \one$, we compute
    \begin{align*}
        \innersmall{S}{\dot X\transpose \dot X} = \innersmall{S}{(X\transpose X) \odot \alpha\alpha\transpose} = \alpha\transpose ( S \odot X\transpose X) \alpha = \alpha\transpose \left( \ddiag(X\transpose X A) - A \odot X\transpose X \right) \alpha.
    \end{align*}
    Since $A$ and $X\transpose X$ are symmetric, it holds that $\diag(X\transpose X A) = (A \odot X\transpose X)\one$.
    Thus, we have found $\innersmall{S}{\dot X\transpose \dot X} = \alpha\transpose L(A \odot X\transpose X) \alpha$ where $L$ is the graph Laplacian~\eqref{eq:laplaciandef}.
    
    Furthermore, $(X\transpose \dot X)_{ij} = \alpha_j x_i\transpose Jx_j^{}$ and $(\dot X\transpose X)_{ij} = \alpha_i (Jx_i^{})\transpose x_j = -\alpha_i x_i\transpose Jx_j^{}$ so that
    \begin{align*}
        (X\transpose \dot X + \dot X\transpose X)_{ij} & = (\alpha_j - \alpha_i) x_i\transpose Jx_j^{}.
    \end{align*}
    Since $J$ rotates vectors in $\reals^2$ by $\pi/2$, it is easy to check that $x_i\transpose J x_j^{}$ is the sine of the angle between $x_i$ and $x_j$, whereas $x_i\transpose x_j^{}$ is the cosine of that angle.
    It then follows from $\sin(\theta)^2 = 1-\cos(\theta)^2$ that $\big(x_i\transpose J x_j^{}\big)^2 = 1 - (x_i\transpose x_j^{})^2$.
    %
    %
    Thus,
    \begin{align*}
        \frac{1}{2} \inner{(X\transpose \dot X + \dot X\transpose X)^{\odot 2}}{W \odot \varphi''(X\transpose X)} = \frac{1}{2} \sum_{i,j} w_{ij} \varphi''(x_i\transpose x_j^{}) \big(1 - (x_i\transpose x_j^{})^2\big) (\alpha_i - \alpha_j)^2.
    \end{align*}
    From the comments after eq.~\eqref{eq:laplaciandef}, we recognize a Laplacian structure.
    We have
    \begin{align*}
        \frac{1}{2} \inner{(X\transpose \dot X + \dot X\transpose X)^{\odot 2}}{W \odot \varphi''(X\transpose X)} = \alpha\transpose L(K) \alpha,
    \end{align*}
    with weights $K_{ij} = w_{ij} \varphi''(x_i\transpose x_j^{}) (1 - (x_i\transpose x_j)^2)$ as in \eqref{eq:A_K_def}.
    Overall, we found
    \begin{align}
        -\innersmall{\dot X}{\Hess f(X)[\dot X]} &= \alpha\transpose \left( \ddiag(X\transpose X A) - A \odot X\transpose X \right) \alpha - \alpha\transpose L(K) \alpha \label{RHSforinequality} \\
        &= \alpha\transpose L(A \odot X\transpose X - K) \alpha = \alpha\transpose L(M) \alpha.\label{eqnforM}
    \end{align}
    The Hessian quadratic form is negative semidefinite if and only if the right-hand side quadratic form~\eqref{RHSforinequality} is nonnegative for all $\alpha \in \Rn$,
    which is equivalent to the matrix inequality in \eqref{eq:socpcircles_cond}.
    From~\eqref{eqnforM}, we also see that the eigenvalues of $\Hess f(X)$ equal those of $-L(M)$.
\end{proof}

\section{Benign landscape in a hemisphere} \label{app:hemisphere}
In this section, we provide a proof of Lemma~\ref{lem:hemisphere},
which we use in the proofs of Theorems~\ref{thm:taylorphi} and~\ref{thm:largebeta}.
\citet[Prop.~12]{markdahl2018nsphere}
showed that if all the points $x_1, \ldots, x_n$ lie in a common \emph{open} hemisphere, then first-order criticality implies global optimality (this is classical for linear $\varphi$).
We improve this slightly to allow a general connected graph and to allow the points to lie in a \emph{closed} hemisphere by using second-order conditions.

\begin{proof}[Proof of Lemma~\ref{lem:hemisphere}]
	Because $\varphi'$ is positive, $X$ is a global maximum if and only if $x_1 = \cdots = x_n$.
	Since $X$ is critical, Lemma~\ref{lem:gradhess} provides $X A - X \ddiag(X\transpose X A) = 0$,
	where $A = W \odot \varphi'(X\transpose X)$.
    (That is, the gradient at $X$ is zero.)
	Transposing and multiplying by $v$ yields, entrywise,
	\begin{align*}
		0 = (AX\transpose v)_i - (X\transpose X A)_{ii} (X\transpose v)_i =
		\sum_{j=1}^n w_{ij} \varphi'(x_i\transpose x_j^{}) \left( (x_j\transpose v) - (x_i\transpose x_j^{})(x_i\transpose v) \right) && \textrm{ for all } i.
	\end{align*}
	Select $i$ such that $x_i\transpose v \geq 0$ is smallest among $\{x_1\transpose v, \ldots, x_n\transpose v\}$.
	Then each term in the sum is nonnegative and therefore must be zero.
	Choose $j$ distinct from $i$ such that $w_{ij} > 0$ (which exists since the graph is connected).
	By the assumption $\varphi'(x_i\transpose x_j^{}) > 0$, we obtain
	\begin{align*}
		x_i\transpose v \leq x_j\transpose v = (x_i\transpose x_j^{})(x_i\transpose v) \leq x_i\transpose v.
	\end{align*}
	Thus, the inequalities are equalities.
	
	If $x_i\transpose v > 0$, we deduce $x_i\transpose x_j^{} = 1$ (using $(x_i\transpose x_j^{})(x_i\transpose v) = x_i\transpose v$), and so $x_i = x_j$.
	Using that $x_j\transpose v = x_i\transpose v$ is minimal among $\{x_1\transpose v, \ldots, x_n\transpose v\}$, we can repeat this argument across a spanning tree of positive weights to conclude that $x_1 = \cdots = x_n$.
	
	Otherwise, if $x_i\transpose v = 0$, then $x_j\transpose v = 0$ as well (since $x_i\transpose v = x_j\transpose v$).
	Repeat this argument across a spanning tree of positive weights to deduce that $X\transpose v = 0$.
	Thus $\rank(X) < d$.
	
	To handle this rank-deficient case, we use second-order conditions, namely, that the Riemannian Hessian is negative semidefinite.
	Assume, without loss of generality, that $v$ is unit-norm (we still have $X \transpose v = 0$).
	Set $\dot X = v\one\transpose \in \Rdn$,
	and note that $X\transpose \dot X = 0$ (which guarantees that $\dot X$ is a valid tangent vector)
	and that $\dot X \transpose \dot X = \one \one\transpose$.
	
	With the Hessian as given in Lemma~\ref{lem:gradhess}, we have,
	again using the fact that $X \transpose \dot X = 0$,
	\begin{align*}
		0 \geq \innersmall{\dot X}{\Hess f(X) [\dot X]} & = - \innersmall{\dot X\transpose \dot X}{S} \\
		&= \innersmall{\one \one\transpose}{A - \ddiag(X \transpose X A) } = \sum_{i,j=1}^n w_{ij} \varphi'(x_i\transpose x_j^{})(1 - x_i\transpose x_j^{}),
	\end{align*}
	where $S = \ddiag(X\transpose X A) - A$ and, again, $A = W \odot \varphi'(X\transpose X)$ as given in \eqref{eq:SA}.
	The sum is over nonnegative terms, hence they must each be equal to zero.
	For $i,j$ such that $w_{ij} > 0$, we deduce from $\varphi'(x_i\transpose x_j^{}) > 0$ that $x_i = x_j$.
	Applying this to a spanning tree of positive weights confirms that $x_1 = \cdots = x_n$.
\end{proof}

\section{Synchronization on spheres ($d \geq 3$): a direct proof of Theorem~\ref{thm:spheres}} \label{app:spheres}
In this section, we provide a proof for Theorem~\ref{thm:spheres}, adapting the proof technique in \citep[Thm.~1.1]{mcrae2023benignlowdim}.
That paper considers synchronization on more general Stiefel manifolds (beyond circles and spheres)
with general connected graphs.
It is limited to what here would be a linear $\varphi$, but the extension to increasing, convex $\varphi$ is easy on spheres.


To exploit second-order criticality conditions, we should perturb the points $x_1, \ldots, x_n$.
Rather than selecting these perturbations deterministically, it is convenient to choose them at random. 
Moreover, as the goal is to achieve synchrony, we perturb the points toward the \emph{same} direction.
Explicitly, with $\gamma \in \Rd$ random, we move $x_i$ in the direction $\dot x_i = \gamma - (x_i\transpose \gamma)x_i$---the projection of $\gamma$ to the tangent space of the sphere at $x_i$.
The same $\gamma$ is used for all points. 
We make this precise below.

Assume $X$ is a critical point for $f$ where the Hessian is negative semidefinite.
Lemma~\ref{lem:gradhess} provides
\begin{align}
	XS = 0 && \textrm{ and } && \innersmall{S}{\dot X\transpose \dot X} \geq 0
	\label{eq:socpphiconvex}
\end{align}
for all $\dot X \in \T_X\calM$, with $S = \ddiag(X\transpose X A) - A$ and $A = W \odot \varphi'(X\transpose X)$;
we have simplified the Hessian expression~\eqref{eq:Hessfquadraticform}, because the assumption $\varphi''(t) \geq 0$ for all $t \in [-1, 1]$ implies that the second term in~\eqref{eq:Hessfquadraticform} is nonnegative.

Since the inequality holds for all tangent $\dot X$, we can also allow $\dot X$ to be random in $\T_X\calM$ and claim the inequality holds in expectation, that is, $\innersmall{S}{\expecttsmall{\dot X\transpose \dot X}} \geq 0$.

Let $\gamma \sim \mathcal{N}(0, I_d)$ be a random vector in $\Rd$ with i.i.d.\ entries following a standard normal distribution.
We use it to build a random tangent vector at $X$ as follows, with $\one \in \Rn$ and $\Proj_X$~\eqref{eq:Proj} the orthogonal projector to the tangent space at $X$:
\begin{align}  
	\dot X & = \Proj_X(\gamma \one\transpose) = \gamma \one\transpose - X\diag(X\transpose \gamma).
\end{align}
Then
\begin{align*}
	\dot X\transpose \dot X & = \|\gamma\|^2 \one \one\transpose - \one \gamma\transpose X \diag(X\transpose \gamma) - \diag(X\transpose \gamma) X\transpose \gamma \one\transpose + \diag(X\transpose \gamma) X\transpose X \diag(X\transpose \gamma) \\
	& = \|\gamma\|^2 \one \one\transpose - \one (\gamma\transpose X)^{\odot 2} - (X\transpose \gamma)^{\odot 2} \one\transpose + (X\transpose X) \odot (X\transpose \gamma \gamma\transpose X).
\end{align*}
Taking expectations, we have $\expectt{\|\gamma\|^2} = d$, $\expectt{\gamma\gamma\transpose} = I_d$ and $\expectt{(X\transpose \gamma)_i^2} = \expectt{x_i\transpose \gamma \gamma\transpose x_i} = x_i\transpose x_i^{} = 1$, so that
\begin{align*}
	\expecttsmall{\dot X\transpose \dot X} & = (d-2) \one \one\transpose + (X\transpose X)^{\odot 2}.
\end{align*}
We decompose the second term as follows, in order to isolate the discrepancy between $x_i\transpose x_j^{}$ and 1 (its target value):
\begin{align*}
	(x_i\transpose x_j^{})^2 = (1 - x_i\transpose x_j^{})^2 - 1 + 2 x_i\transpose x_j^{}.
\end{align*}
In matrix notation, $(X\transpose X)^{\odot 2} = (\one \one\transpose - X\transpose X)^{\odot 2} - \one \one\transpose + 2X\transpose X$.
Therefore,
\begin{align*}
	\expecttsmall{\dot X\transpose \dot X} & = (d-3) \one \one\transpose + (\one \one\transpose - X\transpose X)^{\odot 2} + 2X\transpose X.
\end{align*}
Exploiting the first-order condition $XS = 0$, we further obtain
\begin{align*}
	0 \leq \innerbig{S}{\expecttsmall{\dot X\transpose \dot X}} = (d-3) \innerbig{S}{\one \one\transpose} + \innerbig{S}{(\one \one\transpose - X\transpose X)^{\odot 2}}.
\end{align*}
Recall that $S = \ddiag(X\transpose X A) - A$ with $A = W \odot \varphi'(X\transpose X)$.
Thus
\begin{align*}
	\innersmall{S}{\one \one\transpose} = \trace(X\transpose X A) - \innersmall{A}{\one\one\transpose} = \innersmall{A}{X\transpose X - \one\one\transpose} \leq 0
\end{align*}
owing to the fact that each entry of $X\transpose X$ is in the interval $[-1, 1]$ and the assumption $\varphi'(t) \geq 0$ for $t \in [-1, 1]$ ensures the entries of $A$ are nonnegative.
Moreover, $d-3 \geq 0$, and the diagonal of $\one \one\transpose - X\transpose X$ is zero, so
\begin{align}
	0 & \leq (d-3) \innerbig{S}{\one \one\transpose} + \innerbig{S}{(\one \one\transpose - X\transpose X)^{\odot 2}} \nonumber\\
	& \leq -\innersmall{A}{(\one \one\transpose - X\transpose X)^{\odot 2}} \nonumber\\
	& = - \sum_{i,j} w_{ij} \varphi'(x_i\transpose x_j^{}) (1 - x_i\transpose x_j^{})^2 \leq 0, \label{eq:almostthere}
\end{align}
where the last inequality follows once again from $\varphi'(t) \geq 0$ for all $t \in [-1, 1]$.
As a result, the final sum in~\eqref{eq:almostthere} is equal to zero, so each individual term is equal to zero.
If nodes $i$ and $j$ are connected by an edge ($w_{ij} > 0$), then the (stricter) assumption $\varphi'(t) > 0$ for $t \in [-1, 1]$ forces $1 - x_i\transpose x_j^{} = 0$, that is, $x_i = x_j$.
As we have assumed the graph is connected, apply the same argument along the edges of a spanning tree to deduce that $x_1 = \cdots = x_n$.
(If the graph is not connected, apply the same reasoning to a spanning forest 
to deduce synchrony in each connected component.)
This concludes the proof of Theorem~\ref{thm:spheres}.

\begin{remark}
	To establish~\eqref{eq:socpphiconvex}, we simplified the conclusions of Lemma~\ref{lem:gradhess} by using the assumption $\varphi''(t) \geq 0$.
	That is the only place where that assumption is used.
	Alternatively, we could keep the full expression for the Hessian and compute the expectation of $(X\transpose \dot X + \dot X\transpose X)_{ij}^2$.
	This yields a more refined final inequality which could be used to relax the assumptions on $\varphi$ or the assumption $d \geq 3$.
	However, it is unclear to us how to improve the results for $\varphi_\beta$ with $d = 2$ in this way.
\end{remark}

\section{Proofs from Section~\ref{sec:circles}: Lemmas~\ref{proprepeatingpoints} and~\ref{lemmaIFT}}\label{apprepeatingpoints}

\subsection{Lemma~\ref{proprepeatingpoints}: trading weights for repetitions}

To prove Lemma~\ref{proprepeatingpoints}, we first prove the following finer statement:

\begin{lemma}\label{proprepeatingpointsrefinement}
Let $d=2$, and let $m$ be any positive integer.  Let $q \in \reals^m$ have positive integer entries, let $n = q_1 + \cdots + q_m$, and define the $n \times n$ diagonal matrix
\begin{align*}
  \tilde{Q} = \begin{pmatrix*}
    {q_1} I_{q_1} & & \\
    & \ddots & \\
    & & {q_m} I_{q_m}.
  \end{pmatrix*}.
\end{align*}
Given $X = (x_1, \ldots, x_m) \in \reals^{2 \times m}$, define
\begin{align*}
  \tilde X = (\underbrace{x_1, \ldots, x_1}_\text{$q_1$ times}, \ldots,\underbrace{ x_m, \ldots, x_m}_\text{$q_m$ times}) \in \reals^{2 \times n}.
\end{align*}

If $X$ is critical for $f_m(X) = \frac{1}{2} \langle q q\transpose , \varphi(X\transpose X) \rangle$, then $\tilde X$ is critical for $f_n(\tilde X) = \frac{1}{2} \langle \one_n^{} \one_n\transpose , \varphi({\tilde X}\transpose \tilde X) \rangle$.

Further, let $M$, $m_{ij} = q_i q_j h(x_i\transpose x_j)$, be the adjacency matrix corresponding to $\Hess f_m(X)$, as described in equation~\eqref{hesseqnintheta} of Lemma~\ref{lem:socpcircles}, and define $D = \diag(M \one_m)$.
Likewise define $\tilde M, \tilde D$ for $\Hess f_n(\tilde X)$.

Let $\lambda_1, \ldots, \lambda_{m-1}, 0$ be the eigenvalues of the Laplacian $L(M)$.
Then $\tilde{Q}^{1/2} L(\tilde M) \tilde{Q}^{1/2}$ has eigenvalues $\lambda_1, \ldots, \lambda_{m-1}, 0$ too, and also has eigenvalues $d_{ii}$, $i = 1, \dots, m$, each with multiplicity $q_i-1$ (where $d_{ii}$ is the $i$th diagonal entry of $D$).
This covers all $n$ eigenvalues of $\tilde{Q}^{1/2} L(\tilde M) \tilde{Q}^{1/2}$.
\end{lemma}

\begin{proof}[Proof of Lemma~\ref{proprepeatingpointsrefinement}]
Define $Q = \diag(q) \in \reals^{m \times m}$, and define the $n \times m$ matrix
\[
	E = \begin{pmatrix*}
		\one_{q_1} & 0 & 0 & \ldots & 0\\
		0 & \one_{q_2} & 0 & \ldots & 0\\
		\vdots & \vdots & \vdots & \ddots & \vdots \\
		0 & 0  & 0 & \ldots & \one_{q_m}
	\end{pmatrix*}.
\]
Note the identities
\begin{align} \label{usefuleqnstar1}
\tilde X = X E\transpose, && E\transpose E = Q, && \tilde Q E = E Q.
\end{align}
Following equation~\eqref{eq:A_K_def} of Lemma~\ref{lem:socpcircles} (for both $f_m$ and $f_n$), define $A = q q\transpose \odot \varphi'(X\transpose X) \in \reals^{m \times m}$ and $\tilde A = \varphi'(\tilde X\transpose \tilde X) \in \reals^{n \times n}$.
Defining $A' = \varphi'(X\transpose X)$, we obtain the identities
\begin{align} \label{usefuleqnstar2}
A = Q A' Q, && \tilde A = E A' E\transpose.
\end{align}
Since $X$ is critical for $f_m$, Lemma~\ref{lem:socpcircles} gives $X A = X \ddiag(X\transpose X A)$.
Hence, using equations~\eqref{usefuleqnstar1} and~\eqref{usefuleqnstar2}, we have
\begin{align*}
	\tilde X \ddiag(\tilde X\transpose \tilde X \tilde A) &= X E\transpose \ddiag(E X\transpose X Q A' E\transpose) \\
	&= X E\transpose \ddiag(E X\transpose X A Q^{-1} E\transpose) \\
	&\stackrel{\mathclap{(1)}}{=} X \ddiag(X\transpose X A Q^{-1}) E\transpose \\
	&= X \ddiag(X\transpose X A) Q^{-1} E\transpose\\
	&= X A Q^{-1} E\transpose \\
	&= X Q A' Q Q^{-1} E\transpose \\
	&= X E\transpose E A' E\transpose \\
	&= \tilde X \tilde A,
\end{align*}
where equality $\stackrel{(1)}{=}$ follows from the identity $E\transpose \ddiag(E B E\transpose) = \ddiag(B) E\transpose$ for $B \in \reals^{m \times m}$.
We conclude that $\tilde X$ is critical for $f_n$, again by Lemma~\ref{lem:socpcircles}.

Let us move on to the second-order condition.
Our goal is to express the eigenvalues of $L(\tilde M) = \tilde D - \tilde M$ in terms of those of $L(M) = D - M$.
Towards this end, define the $m \times m$ matrices $M' = Q^{-1} M Q^{-1}$ and $D' = Q^{-1} D Q^{-1}$, which have entries
\begin{align*}
    m_{ij}' = (q_i q_j)^{-1} m_{ij} = h(x_i \transpose x_j) && \textrm{ and } && d_{ii}' = q_i^{-2} d_{ii}.
\end{align*}
By the definition of $\tilde M$, we have
\[
	\tilde M = \begin{pmatrix*}
		m_{11}' \one_{q_1}^{} \one_{q_1}\transpose & m_{12}' \one_{q_1}^{} \one_{q_2}\transpose  & \ldots & m_{1 m}' \one_{q_1}^{} \one_{q_m}\transpose\\
		\vdots & \vdots & \ddots & \vdots \\
		m_{m1}' \one_{q_m}^{} \one_{q_1}\transpose & m_{m 2}' \one_{q_m}^{} \one_{q_2}\transpose  & \ldots & m_{m m}' \one_{q_m}^{} \one_{q_m}\transpose
	\end{pmatrix*} = E M' E\transpose.
\]
Moreover, $\tilde D = \diag(\tilde M \one_n)$ is diagonal
with diagonal blocks $\tilde D_i$ given by:
\begin{align}\label{formoftildeD}
	\tilde D = \begin{pmatrix*}
		\tilde D_1 &  & \\
		 & \ddots & \\
		 &  & \tilde D_m
	\end{pmatrix*},
	\quad \quad \quad
	\tilde D_i 
	= q_i^{-1} d_{ii} I_{q_i} = q_i d_{ii}' I_{q_i}.
\end{align}

We can decompose $L(\tilde M)$ as
\begin{align*}
L(\tilde M) &= \tilde D - \tilde M  \\
&= [\tilde D - E D' E\transpose]  + E(D' - M')E\transpose
\\ &=  [\tilde D - E D' E\transpose]  + E Q^{-1}L(M)Q^{-1}E\transpose.
\end{align*}
Using $\tilde Q E = E Q$, we get
\begin{align}\label{decompoositionofLMprime}
\tilde{Q}^{1/2} L(\tilde M) \tilde{Q}^{1/2} =  \underbrace{\tilde{Q}^{1/2}[\tilde D - E D' E\transpose]\tilde{Q}^{1/2}}_{\text{left term}}  
+ \underbrace{E Q^{-1/2} L(M)Q^{-1/2} E\transpose}_{\text{right term}}.
\end{align}
Crucially, the decomposition~\eqref{decompoositionofLMprime} reveals the eigendecomposition of $\tilde{Q}^{1/2} L(\tilde M) \tilde{Q}^{1/2}$.
Indeed, $E Q^{-1/2}$ is orthonormal since $E\transpose E = Q$, and the column spaces of the left and right terms of~\eqref{decompoositionofLMprime} are orthogonal.
This latter observation is apparent from the fact that $\tilde{Q}^{1/2}[\tilde D - E D' E\transpose]\tilde{Q}^{1/2}$ is \emph{block diagonal} with $i$-th diagonal block given by
\begin{align}\label{formofthediagonalblocks}
q_i \big(\tilde D_i - d_{ii}' \one_{q_i}^{} \one_{q_i}\transpose\big) = d_{ii}\big(I_{q_i} - q_i^{-1} \one_{q_i}^{} \one_{q_i}\transpose\big),
\end{align}
where we have used~\eqref{formoftildeD}.

We conclude that $\tilde{Q}^{1/2} L(\tilde M) \tilde{Q}^{1/2}$ has $m$ eigenvalues equal to the eigenvalues of $L(M)$, due to the right term of~\eqref{decompoositionofLMprime}.
Due to the left term of~\eqref{decompoositionofLMprime} and the form of the blocks~\eqref{formofthediagonalblocks}, $\tilde{Q}^{1/2} L(\tilde M) \tilde{Q}^{1/2}$ also has eigenvalues $d_{ii}, i = 1, \ldots, m,$ each with multiplicity $q_i-1$.

To summarize, we have identified all
$m + \sum_{i=1}^m (q_i  - 1) = n$
eigenvalues of $\tilde{Q}^{1/2} L(\tilde M) \tilde{Q}^{1/2}$.
\end{proof}

Lemma~\ref{proprepeatingpoints} now follows as a special case of Lemma~\ref{proprepeatingpointsrefinement}:

\begin{proof}[Proof of Lemma~\ref{proprepeatingpoints}]
We use the notation from Lemma~\ref{proprepeatingpointsrefinement}.
Since we assume $\Hess f_m(X)$ is negative definite (up to the trivial eigenvalue due to symmetry, as usual), 
Lemma~\ref{lem:socpcircles} implies $\lambda_{i} > 0$ for $i=1, \ldots, m-1,$ and $D - M \succeq 0$.
Thus,
\[
	d_{ii} \geq m_{ii} = q_i^2 h(1) = q_i^2 \varphi'(1) > 0.
\]
So invoking Lemma~\ref{proprepeatingpointsrefinement}, we have shown that $\tilde{Q}^{1/2} L(\tilde M) \tilde{Q}^{1/2}$ has 
\[
	(m-1) + \sum_{i=1}^m (q_i  - 1) = n - 1
\]
positive eigenvalues (and of course has one remaining zero eigenvalue).
So the same is true for $L(\tilde M)$.
Lemma~\ref{lem:socpcircles} then tells us that $\Hess f_n(\tilde X)$ is also negative definite.
\end{proof}

\subsection{Lemma~\ref{lemmaIFT}: strict spurious points do not vanish under small perturbations}
\begin{proof}[Proof of Lemma~\ref{lemmaIFT}]
In order to ensure that $\varepsilon$ is nonnegative, we reparameterize $\varepsilon = \alpha^2$ for $\alpha \in \reals$.
Consider the function
\[
	f_{q, \alpha, \tau}(X) := \frac{1}{2}\langle q q\transpose, \varphi_{\alpha^2, \tau}(X\transpose X)\rangle.
\]
We think of $f_{q, \alpha, \tau}$ as a perturbation of the function $f_{\one_m, 0, 0}$, which is simply $f$ with $W$ the all-ones matrix and the ReLU $\varphi(t) = \max\{0, t\}$.
Let $X$ be the regular $m$-gon~\eqref{eq:ngon}.

We want to apply the implicit function theorem \cite[Thm.~3.3.1]{krantz2013implicitfunctionthm} to the map
\[
	F \colon ((q, \alpha, \tau), Y) \mapsto \grad f_{q, \alpha, \tau}(Y)
\]
at $q = \one_m, \alpha = 0, \tau = 0$ and $Y=X$.
In order to do this, we need:
\begin{itemize}
\item[(a)] $F$ to be continuously differentiable in a neighborhood of $q = \one_m, \alpha = 0, \tau = 0$, $Y=X$.
\item[(b)] $0 = F((\one_m,0, 0), X) = \grad f_{\one_m,0, 0}(X)$.  This is true by Lemma~\ref{propphiforrelu} (and the assumption $m$ is odd).
\item[(c)] The differential of $Y \mapsto F((\one_m, 0, 0), Y)$ at $Y=X$ to be invertible.  That differential is exactly $\Hess f_{\one_m, 0, 0}(X)$, which is negative definite (and so invertible) by Lemma~\ref{propphiforrelu}.\footnote{Strictly speaking, the Hessian is not negative definite since it has a single zero eigenvalue due to the global rotation symmetry.  So to apply the implicit function theorem we must first mod out that zero eigenvalue, e.g., by fixing one of the points or passing to the quotient.  Then the Hessian becomes truly negative definite.}
\end{itemize}
For item (a): by Lemma~\ref{lem:gradhess},
$\varphi_{\alpha^2, \tau}$ appears in $F$ only through $\varphi_{\alpha^2, \tau}'$.
So it is enough to verify that
\[
	(\alpha, \tau, t) \mapsto \varphi_{\alpha^2, \tau}'(t) = 1 - \frac{1}{1 + e^{(t-\tau)/\alpha^2}}
\]
is continuously differentiable in a neighborhood of $(0, 0, t_0)$ when $t_0 \neq \tau$.\footnote{Since $m$ is odd, none of the inner products $x_i\transpose x_j^{}$ of the regular $m$-gon lie on the kink of the ReLU; this is why we only require continuous differentiability when $t_0 \neq \tau$.}
It is straightforward to do this by explicitly computing the differential of $(\alpha, \tau, t) \mapsto \varphi_{\alpha^2, \tau}'(t)$; we omit the details.

The hypotheses of the implicit function theorem are satisfied, and that theorem yields that there is a $\delta > 0$ such that if $\|q - \one_m\| \leq \delta$, $\alpha \in [-\sqrt{\delta}, \sqrt{\delta}]$ and $\tau \in [-\delta, \delta]$, 
then there is an $X_{q, \alpha, \tau} \in \calM$ \emph{near} the $m$-gon $X$ such that 
\[
	0 = F((q, \alpha, \tau), X_{q, \alpha, \tau}) = \grad f_{q, \alpha, \tau}(X_{q, \alpha, \tau}).
\]
Since $X_{q, \alpha, \tau}$ is near $X$, by continuity $X_{q, \alpha, \tau}$ is also spurious and has negative definite Hessian (except for the single zero eigenvalue), possibly after making $\delta$ smaller.
\end{proof}

\section{Synchronization with $\beta \geq \frac{n^2}{\pi^2}$: Theorem~\ref{thm:largebeta}} \label{sec:largebeta}

We now give a proof for Theorem~\ref{thm:largebeta}.
That result and the proof below are due to \citet{geshkovski2025mathtransformers}, with only slight improvements as outlined in the introduction.

\begin{proof}[Proof of Theorem~\ref{thm:largebeta}]
Assume $X$ is critical for $f$ with negative semidefinite Hessian.
From Lemma~\ref{lem:socpcircles}, we know the Laplacian~\eqref{eq:laplaciandef} $L(M) = \diag(M\one) - M$ is positive semidefinite, with
\begin{align*}
    m_{ij} = w_{ij} h(x_i\transpose x_j^{}) && \textrm{ and } &&
    h(t) = t \varphi'(t) - (1-t^2) \varphi''(t).
\end{align*}
By condition~\eqref{conditiononvarphiforlargebeta}, we know that $h(t) \geq 0$ implies $t \geq \cos(\frac{\pi}{n})$.

Pick a nonempty proper subset $S \subset \{1, \ldots, n\}$.
Let $\alpha_i = 1$ if $i \in S$ and $\alpha_i = 0$ if $i \notin S$.
Since the Laplacian $L$ is positive semidefinite we find
\begin{align*}
  \alpha\transpose L \alpha = \frac{1}{2} \sum_{i,j} m_{ij}(\alpha_i - \alpha_j)^2 = \sum_{i \in S, j \notin S} w_{ij} h(x_i\transpose x_j^{}) \geq 0.
\end{align*}
There exist indices $i \in S, j \notin S$ such that $w_{ij} > 0$
(as otherwise all the weights between $S$ and its complement would be zero, and the graph would be disconnected).
For at least one of those pairs, $h(x_i\transpose x_j^{}) \geq 0$ (as otherwise the sum would be negative), i.e., $x_i\transpose x_j^{} \geq \cos(\frac{\pi}{n})$ by condition~\eqref{conditiononvarphiforlargebeta}.
With $\dist$ denoting the distance on the circle (in radians), this gives $\dist(x_i, x_j) \leq \cos^{-1}\!\big(\cos(\frac{\pi}{n})\big) = \frac{\pi}{n}$.

Starting with $S = \{1\}$, apply the argument above to identify a node in the complement.
Add it to $S$ and repeat.
This gradually grows a spanning tree over the $n$ nodes, and it satisfies $\dist(x_i, x_j) \leq \frac{\pi}{n}$ for each edge of the tree.
The radius of a tree is at most $n/2$, hence we can select a central node $x_i$ such that $\dist(x_i, x_j) \leq \frac{n}{2} \cdot \frac{\pi}{n} = \frac{\pi}{2}$ for all $j$.
In other words, $x_1, \ldots, x_n$ lie in the same half circle: we can then apply Lemma~\ref{lem:hemisphere} to conclude, proving the main part of Theorem~\ref{thm:largebeta}.

It remains to verify that if $\varphi''(t) \geq \frac{n^2}{\pi^2} \varphi'(t) > 0$ for all $t$, then condition~\eqref{conditiononvarphiforlargebeta} holds.  
We can of course assume $n \geq 2$ (because if $n=1$, the state is already trivially synchronized).
Assume $h(t) \geq 0$; we want to show $t \geq \cos(\frac{\pi}{n})$.
First, note that the assumption on $\varphi$ gives
\[
	0 \leq h(t) = t \varphi'(t) - (1-t^2) \varphi''(t) \leq \Big(t - \frac{n^2}{\pi^2}(1-t^2)\Big) \varphi'(t).
\]
Since $\varphi' > 0$, we conclude $t - \frac{n^2}{\pi^2}(1-t^2) \geq 0$, which implies\footnote{We already have $t \geq 0$ (as otherwise $h(t) < 0$), and we can assume $t < 1$ (as otherwise we are done).}
$\frac{t}{1-t^2} \geq \frac{n^2}{\pi^2}$.
Using the inequality $\frac{\cos(s)}{1 - \cos(s)^2} \leq \frac{1}{s^2}$ for all $s \in (0, \frac{\pi}{2}]$, we have $\frac{t}{1-t^2} \geq \frac{n^2}{\pi^2} \geq \frac{\cos(\pi/n)}{1 - \cos(\pi/n)^2}$.
This implies $t \geq \cos(\frac{\pi}{n})$, as $t \mapsto \frac{t}{1-t^2}$ is increasing on $[-1,1]$.
\end{proof}
%

\section{The quadratic case $\varphi(t) = \frac{1}{2}t^2$} \label{sec:quadraticphi}


We prove Theorem~\ref{thm:quadraticphi} in this section via three lemmas.

\begin{lemma} \label{lem:gradhessquadratic}
    Within the context of Theorem~\ref{thm:quadraticphi},
    if $X \in \calM$ is a critical point of $f$ then
    \begin{align}
        (XX\transpose)X = XD && \textrm{ where } && D = \ddiag(X\transpose X X\transpose X).
        \label{eq:focpsquared}
    \end{align}
    If the Hessian at $X$ is negative semidefinite then, for all $\dot X \in \T_X\calM$,
    \begin{align}
        \innersmall{\dot X\transpose \dot X}{\ddiag(X\transpose X X\transpose X)} & \geq \innerbig{XX\transpose}{\dot X \dot X\transpose} + \innerbig{X\dot X\transpose}{\dot X X\transpose} + \sqfrobnormbig{\dot X X\transpose}.
        \label{eq:socptsquared}
    \end{align}
\end{lemma}
\begin{proof}
    Consider Lemma~\ref{lem:gradhess} with $\varphi'(t) = t$ and $\varphi''(t) = 1$.
    We see that $A = X\transpose X$ and $S = \ddiag(X\transpose X X\transpose X) - X\transpose X$.
    The first-order condition $XS = 0$ provides~\eqref{eq:focpsquared}.
    The second-order conditions provide
    \begin{align*}
        0 & \geq - \innersmall{\dot X\transpose \dot X}{S} + \frac{1}{2} \sqfrobnormbig{X\transpose \dot X + \dot X\transpose X}
    \end{align*}
    for all $\dot X$ in the tangent space at $X$.
    The right-most term expands into
    \begin{align*}
        \frac{1}{2}\sqfrobnormbig{X\transpose \dot X + \dot X\transpose X} & = \frac{1}{2}\innerbig{X\transpose \dot X + \dot X\transpose X}{X\transpose \dot X + \dot X\transpose X} = \innerbig{XX\transpose}{\dot X \dot X\transpose} + \innerbig{X\dot X\transpose}{\dot X X\transpose},
    \end{align*}
    whereas $\innersmall{\dot X\transpose \dot X}{S} = \innersmall{\dot X\transpose \dot X}{\ddiag(X\transpose X X\transpose X)} - \innersmall{\dot X\transpose \dot X}{X\transpose X}$.
    Combine to confirm~\eqref{eq:socptsquared}.
\end{proof}

The first-order condition~\eqref{eq:focpsquared} is formatted to highlight the following key fact: each column $x_i$ of $X$ is an eigenvector of $XX\transpose \in \Rdd$, with eigenvalue $D_{ii}$.
Since $XX\transpose$ is symmetric, the spectral theorem tells us that if $x_i, x_j$ are associated to distinct eigenvalues ($D_{ii} \neq D_{jj}$) then $x_i$ and $x_j$ are orthogonal.
The next step is to show that this does not happen.

\begin{lemma}
    With $X$ as in Lemma~\ref{lem:gradhessquadratic}, no two columns of $X$ are orthogonal.
\end{lemma}
\begin{proof}
    For contradiction, assume $x_i\transpose x_j^{} = 0$ for some $i,j$.
    Then, $x_j$ is in the tangent space to the sphere at $x_i$, and vice versa.
    Consequently,
    \begin{align*}
        \dot X = x_ie_j\transpose + x_je_i\transpose
    \end{align*}
    is in the tangent space to $\calM$ at $X$, where $e_k \in \Rn$ is the $k$th column of the identity matrix of size $n$.
    By computation, $\dot X\transpose \dot X = e_i^{} e_i\transpose + e_j^{} e_j\transpose$.
    Also, $\dot X \dot X\transpose = x_i^{} x_i\transpose + x_j^{} x_j\transpose$ and $X\dot X\transpose = x_i^{}x_j\transpose + x_j^{}x_i\transpose = \dot X X\transpose$.
    Plugging these into the second-order conditions~\eqref{eq:socptsquared} reveals that
    \begin{align*}
        \innerbig{e_i^{} e_i\transpose + e_j^{} e_j\transpose}{\ddiag(X\transpose X X\transpose X)} & \geq \innerbig{XX\transpose}{x_i^{} x_i\transpose + x_j^{} x_j\transpose} + 2\sqfrobnormbig{x_i^{}x_j\transpose + x_j^{}x_i\transpose}.
    \end{align*}
    The left-hand side and the first term on the right-hand side are equal.
    Also, $\sqfrobnormbig{x_i^{}x_j\transpose + x_j^{}x_i\transpose} = 2$ since $x_i$ and $x_j$ are orthogonal.
    Overall, we have found $0 \geq 4$: a contradiction indeed.
\end{proof}

Thus, we know that $x_1, \ldots, x_n$ are all eigenvectors for $XX\transpose$ with the \emph{same} eigenvalue.
If those vectors span all of $\Rd$, this imposes severe constraints on $XX\transpose$---too strong, in fact.

\begin{lemma} \label{lem:quadraticrankdef}
    With $X$ as in Lemma~\ref{lem:gradhessquadratic}, it holds that $\rank(X) < d$.
\end{lemma}
\begin{proof}
    For contradiction, assume $\rank(X) = d$.
    Then, we may select $d$ linearly independent columns among those of $X$.
    These form a basis for $\Rd$ and (by the reasoning above) are eigenvectors of $XX\transpose$ for the same eigenvalue $\lambda$.
    It follows that $XX\transpose$ has a single eigenvalue, and so $XX\transpose = \lambda I_d$.
    %
    %
    Plug this and $\diag(X\transpose X) = \one$ into the second-order conditions~\eqref{eq:socptsquared} to reveal that
    \begin{align}
        \lambda \sqfrobnormbig{\dot X} & \geq \lambda \sqfrobnormbig{\dot X} + \innerbig{X\dot X\transpose}{\dot X X\transpose} + \sqfrobnormbig{\dot X X\transpose}
    \end{align}
    for all tangent $\dot X$.
    To contradict this, let $\dot X = ue_1\transpose$ where $e_1$ is the first column of the identity matrix $I_n$, and $u$ is an arbitrary unit-norm vector in the tangent space to the sphere at $x_1$.
    (Such a vector exists as long as $d \geq 2$.)
    Then, $\dot X X\transpose = ux_1\transpose$.
    Since $x_1\transpose u = 0$, the above inequality becomes $0 \geq 1$: a contradiction indeed.
\end{proof}

We now know that $X$ is rank deficient, and we wish to deduce that $X$ is optimal.
Notice that we cannot use Lemma~\ref{lem:hemisphere} here, since $\varphi'(t) = t$ is not positive on $[-1, 1]$.
Thus, we resort to a different argument.

Rotate the points such that the last row of $X$ is zero (formally: let $X = U\Sigma V\transpose$ be an SVD of $X$, and apply $U\transpose$ to $X$).
The new $X$ still satisfies first- and second-order conditions since $f$ is invariant to such rotations.
For now, discard the last row of $X$, producing a new matrix $\tilde X \in \reals^{(d-1) \times n}$.
In particular, $X\transpose X = \tilde X\transpose \tilde X$.
This new matrix also satisfies first- and second-order conditions for $f$ \emph{in the new dimension $d-1$} (because all we did was remove potential directions for improvement).
%
%
%
Moreover, the rank did not change: $\rank(\tilde X) = \rank(X)$.
If $d - 1 \geq 2$, Lemma~\ref{lem:quadraticrankdef} applies and so $\tilde X$ too is rank deficient.
We may repeat this operation until we reach the conclusion that the original $X$ must have had rank 1.

That implies that all columns of $X$ are colinear, hence that they are equal up to sign.
This is indeed maximal for $f$ since $\varphi(t)$ is maximal if and only if $t = \pm 1$: the proof of Theorem~\ref{thm:quadraticphi} is complete.

\section{Matlab script for numerical exploration} \label{sec:matlabcode}

\begin{lstlisting}[style=Matlab-editor]
clear; clf; clc;
% This script requires Manopt 7.0 or later: see manopt.org.

%% Choose dimension d and number of points n
d = 2;
n = 16;

%% Choose the nonlinearity varphi(t)
phichoice = 'expbeta';
switch phichoice
    case 'expbeta'
        beta = 6;
        phi = @(t) exp(beta*(t-1));
        dphi = @(t) beta*exp(beta*(t-1));
        d2phi = @(t) beta^2*exp(beta*(t-1));
    case 'linear'
        phi = @(t) t;
        dphi = @(t) 1;
        d2phi = @(t) 0;
    case 'quadratic'
        phi = @(t) (1/2)*t.^2;
        dphi = @(t) t;
        d2phi = @(t) 1;
    case 'cubic'
        phi = @(t) (1/3)*t.^3;
        dphi = @(t) t.^2;
        d2phi = @(t) 2*t;
    case 'logsumexp'
        m = n; % set to n or to same as q in the qgon below
        tau = cos((3/2)*(2*pi/m));
        eps = 1/m^3;
        phi = @(t) eps*log(1 + exp((t-tau)/eps));
        dphi = @(t) 1./(1 + exp(-(t-tau)/eps));
        d2phi = @(t) (sech((t-tau)/(2*eps)).^2)./(4*eps);
end

%% Choose an adjacency matrix W for the graph
graph = 'complete';
switch graph
    case 'complete'
        W = ones(n, n);
    case 'cycle'
        I = 1:n;
        J = circshift(I, 1);
        W = zeros(n, n);
        W(sub2ind([n, n], I, J)) = 1;
        W = W+W';
end

%% Build the Manopt problem structure
% Oblique factory is Cartesian product of n spheres in R^d.
problem.M = obliquefactory(d, n, false, false);
% Manopt minimizes so the cost function is -f (so as to maximize f).
sgn = -1; % set to -1 to maximize f, and to +1 to minimize f
inner = @(A, B) A(:).'*B(:);
problem.cost = @(X) sgn*(1/n^2)*inner(W, phi(X'*X));
if exist('dphi', 'var') && exist('d2phi', 'var')
    problem.egrad = @(X) sgn*(2/n^2)*X*(W.*dphi(X'*X));
    problem.ehess = @(X, Xdot) ...
                     sgn*(2/n^2)*(Xdot*(W.*dphi(X'*X)) + ...
                     X*(W.*d2phi(X'*X).*(Xdot'*X + X'*Xdot)));
else
    % If dphi, d2phi not provided, try Auto Differentiation.
    problem = manoptAD(problem);
end
% checkgradient(problem);
% checkhessian(problem);

%% Pick and plot an initial configuration of points
init = 'random';
switch init
    case 'random'
        X0 = problem.M.rand();
    case 'ngon'
        X0 = zeros(d, n);
        t = linspace(0, 2*pi, n+1);
        t = t(1:end-1);
        X0(1, :) = cos(t);
        X0(2, :) = sin(t);
    case 'qngon'
        X0 = zeros(d, n);
        q = 5;
        t = linspace(0, 2*pi, q+1);
        t = repmat(t(1:q), 1, ceil(n/q));
        X0(1, :) = cos(t(1:n));
        X0(2, :) = sin(t(1:n));
    case 'tetrahedron'
        assert(d == 3 && n == 4);
        X0 = [1       -1/3       -1/3       -1/3
              0  sqrt(8/9) -sqrt(2/9) -sqrt(2/9)
              0          0  sqrt(6/9) -sqrt(6/9)];
end

% Plot X0
hdots = plot_dots(X0); title('Initial X0');

%% Define some options for optimizers
options.maxiter = 300;
options.tolgradnorm = 1e-12;
% Plot the points at each iteration with a slight pause
options.statsfun = @(problem, X, stats) ...
                      update_dots(hdots, X, .1, stats);

%% Optimize
optimizationplots = 'gradientdynamics';
switch optimizationplots
    case 'fast'
        X = trustregions(problem, X0, options);
    case 'gradientdynamics'
        X = steepestdescent(problem, X0, options);
    case 'none'
        X = X0;
end
update_dots(hdots, X); title('Final X');

%% Check gradient and Hessian at computed point
gnorm = problem.M.norm(X, getGradient(problem, X));
hspec = hessianspectrum(problem, X);
fprintf('Gradient norm at X: %.4g\n', gnorm);
fprintf('Eigenvalues of Hessian at X:\n');
disp(hspec');

%% Helper functions to plot x_1, ..., x_n for d = 2 or 3.
function hdots = plot_dots(X)
    hold all;
    dotscolors = X(1, :) + X(2, :);
    if size(X, 1) == 3
        [Sx, Sy, Sz] = sphere(50);
        surf(Sx, Sy, Sz, 'FaceAlpha', .25, ...
                         'FaceColor', 2*[.1, .2, .3]);
        hdots = scatter3(X(1, :)', X(2, :)', X(3, :)', ...
                         200, dotscolors, 'filled');
    elseif size(X, 1) == 2
        t = linspace(0, 2*pi, 501);
        plot(cos(t), sin(t), 'k-', 'LineWidth', 2);
        hdots = scatter(X(1, :)', X(2, :)', 200, ...
                        dotscolors, 'filled');
    end
    axis equal; axis off;
    set(gcf, 'Color', 'w');
end

function stats = update_dots(hdots, X, time, stats)
    set(hdots, 'XData', X(1, :));
    set(hdots, 'YData', X(2, :));
    if size(X, 1) == 3
        set(hdots, 'ZData', X(3, :));
    end
    drawnow;
    if exist('time', 'var') && ~isempty(time)
        pause(time);
    end
    if ~exist('stats', 'var') % For Manopt's statsfun
        stats = [];
    end
end

\end{lstlisting}

\end{document}